\documentclass[11pt]{article}

\usepackage{amsmath,amsthm,amsfonts}

\parskip=\medskipamount
\newtheorem{thm}{\bf Theorem}

\def\fpd#1#2{{\displaystyle\frac{\partial #1}{\partial #2}}}
\def\spd#1#2#3{{\displaystyle\frac{\partial^2 #1}
{\partial #2\partial #3}}}

\def\vf#1{\frac{\partial}{\partial #1}}

\def\clift#1{#1^{\scriptscriptstyle{\mathrm{C}}}}
\def\hlift#1{#1^{\scriptscriptstyle{\mathrm{H}}}}
\def\vlift#1{#1^{\scriptscriptstyle{\mathrm{V}}}}

\def\A{\mathcal{A}}
\def\E{\mathcal{E}}
\def\R{\mathbf{R}}

\def\onehalf{{\textstyle\frac12}}
\def\threehalf{{\textstyle\frac32}}

\def\oneq{{\textstyle\frac14}}
\def\threeq{{\textstyle\frac34}}

\font\frak=eufm10 scaled\magstep1
\def\goth#1{\hbox{{\frak #1}}}
\def\g{\goth{g}}
\def\la{\g}

\def\cinfty#1{C^{\scriptscriptstyle\infty}(#1)}
\def\vectorfields#1{\goth{X}(#1)}

\def\DV#1{\vlift{\mathrm{D}}_{#1}}

\def\Sym{\mathop{\mathrm{Sym}}}
\def\tr{\mathop{\mathrm{tr}}}
\def\Ad{\mathop{\mathrm{ad}}\nolimits}
\def\ad{\Ad}

\def\lequiv{\lbrack\!\lbrack}
\def\requiv{\rbrack\!\rbrack}
\def\lLAB{\lbrack\!\lbrack}
\def\rLAB{\rbrack\!\rbrack}
\def\LABP#1{\lLAB#1\rLAB^1}
\def\LAB#1{\lLAB#1\rLAB}

\begin{document}

\title{The inverse problem for invariant
Lagrangians on a Lie group}

\author{M.\ Crampin${}^{a}$ and T.\ Mestdag${}^{a,b}$\\[2mm]
{\small ${}^a$Department of Mathematical Physics and Astronomy, Ghent University}\\
{\small Krijgslaan 281, B-9000 Ghent, Belgium}\\[1mm]
{\small ${}^b$ Department of Mathematics, The University of Michigan}\\
{\small 530 Church Street, Ann Arbor, MI 48109, USA}}

\maketitle

{\small {\bf Abstract.} We discuss the problem of the existence of a
regular invariant Lagrangian for a given system of invariant
second-order differential equations on a Lie group $G$, using
approaches based on the Helmholtz conditions.  Although we deal with
the problem directly on $TG$, our main result relies on a reduction
of the system on $TG$ to a system on the Lie algebra of $G$.  We
conclude with some illustrative examples.}



\section{Introduction}

The inverse problem of the calculus of variations consists in
finding conditions for the existence of a regular Lagrangian for a
given set of second-order ordinary differential equations on a
manifold, ${\ddot q}^i=f^i(q,\dot q)$, so that the given equations
are equivalent to the Euler-Lagrange equations of the Lagrangian. In
order for a Lagrangian $L(q,\dot{q})$ to exist we must be able to
find $g_{ij}(q,\dot q)$, so-called multipliers, such that
\[
g_{ij}({\ddot q}^j-f^j) = \frac{d}{dt}\left( \fpd{L}{{\dot q}^i}
\right)-\fpd{L}{q^i}.
\]
It is shown for example in \cite{Douglas,Santilli} that the
multipliers must satisfy
\begin{eqnarray*}
&&\det(g_{ij})\neq 0,\qquad\quad g_{ji}=g_{ij},\\&&
\frac{d}{dt}(g_{ij})+\onehalf \fpd{f^k}{{\dot q}^j}g_{ik}+ \onehalf
\fpd{f^k}{{\dot q}^i}g_{kj}=0,\\ && g_{ik}\left(
\frac{d}{dt}\left(\fpd{f^k}{{\dot
q}^j}\right)-2\fpd{f^k}{q^j}-\onehalf\fpd{f^l}{{\dot
q}^j}\fpd{f^k}{{\dot q}^l} \right) = g_{jk}\left(
\frac{d}{dt}\left(\fpd{f^k}{{\dot
q}^i}\right)-2\fpd{f^k}{q^i}-\onehalf\fpd{f^l}{{\dot
q}^i}\fpd{f^k}{{\dot q}^l} \right),\\&& \fpd{g_{ij}}{{\dot
q}^k}=\fpd{g_{ik}}{{\dot q}^j};
\end{eqnarray*}
and conversely, if one can find functions satisfying these conditions
then the equations $\ddot{q}^i=f^i$ are derivable from a Lagrangian.
These conditions are generally referred to as the Helmholtz
conditions.  The solution $(g_{ij})$ is the Hessian of the sought-for
Lagrangian with respect to the velocity variables, and
$\det(g_{ij})\neq 0$ is the condition for the Lagrangian to be
regular.  We refer to the recent survey \cite{KP} and the monograph
\cite{AndThomp} for comments on the history of the problem, milestones
in the literature and accounts of the different paths that have been
followed in the past.

We will focus here on the case where the manifold is a Lie group.  An
immediate example is the one where the second-order system is the
geodesic spray of the canonical connection on the group:\ this
connection is specified in terms of left-invariant vector fields $X$
and $Y$ by $\nabla_X Y =\onehalf[X,Y]$.  The inverse problem for this
specific type of second-order system has been solved explicitly for
almost all Lie groups up to dimension 6 by Thompson and his
collaborators (see \cite{Gerard4dim,Thompson2,Thompson1} and the
references therein):\ in each case the authors were able to decide if
a Lagrangian exists or not, and to provide a Lagrangian in the
affirmative cases.

The second-order equations of the canonical connection are invariant
under left translations.  Surprisingly, if a Lagrangian exists, it is
not necessarily invariant.  The main goal of this paper is to solve a
type of inverse problem which is on the one hand broader than that
discussed in \cite{Thompson1} etc.\ in that it deals with any
invariant system of second-order ordinary differential equations on a
Lie group, but on the other hand more restricted in that the
Lagrangian, if it exists, is required to be invariant also.  That is
to say, we will deal with the following rather natural problem:\ given
an invariant second-order system on a Lie group $G$, when does there
exist a regular Lagrangian for it that is also invariant under $G$?
We call this the invariant inverse problem.  The invariant inverse
problem for the specific case of the geodesic spray of the canonical
connection has been studied in \cite{Zolt}.  In the current paper, by
contrast, we will deal with the general invariant inverse problem.

It is unfortunately not straightforward to adapt the solution of the
inverse problem by the Helmholtz conditions to the invariant inverse
problem.  Clearly, if there is an invariant Lagrangian then the
corresponding multiplier matrix (its Hessian) must itself be
invariant (in an appropriate sense).  The difficulty is this:\ one
may find a multiplier which satisfies the Helmholtz conditions and
is invariant; one is then guaranteed that there is a Lagrangian, but
not that the Lagrangian is invariant.  Roughly speaking, to obtain
the Lagrangian one has to integrate the multiplier, and invariance
may be lost as a result.  In fact extra conditions, of a
cohomological nature, must be satisfied.  The occurrence of such
cohomological conditions was discussed already nearly twenty years
ago, in a similar but more limited context, by Marmo and
Morandi~\cite{Marmo2}.  We will present a version of their result,
which amounts in fact to a small generalization of it, in
Theorem~\ref{thm1}. The conditions also appear in \cite{AndPoh},
using the rather different framework of the variational bicomplex.

It is however possible to adopt quite a different approach from these
authors, by taking advantage of invariance to carry out a reduction of
the problem, which turns out to simplify it in concept, and to make
the solution considerably more useful in applications.  We next
explain this alternative approach in a little more detail.

Because of the invariance of our problem, $G$ will be a symmetry
group of the second-order system.  It follows that the space of
interest is effectively the Lie algebra $\g$ of $G$ rather than the
whole manifold $TG$, and we can first perform a reduction.  The
dynamical vector field $\Gamma$ corresponding to the system of
differential equations, namely
\[
\Gamma=
{\dot q}^i \fpd{}{q^i} + f^i \fpd{}{{\dot q}^i}\in\vectorfields{TG},
\]
reduces to a vector field $\gamma$ on $\g$ given in terms of
Cartesian coordinates $(w^i)$ (so that the $w^i$ are the components of
$w\in\g$ with respect to some chosen basis of $\g$) by
\[
\gamma=\gamma^i\fpd{}{w^i}.
\]
On the other hand, if $L \in \cinfty{TG}$ is a
regular invariant Lagrangian then its restriction to $\g=T_eG$ is
a function (also called a Lagrangian) $l\in\cinfty{\g}$.  We will take
optimal advantage of the following observation (which is proved in
\cite{MR} for example, though we will give a different derivation
below):\ finding a solution $g(t)\in G$ of the Euler-Lagrange
equations of $L$ is equivalent to finding a solution $w(t)\in \g$ of the
so-called Euler-Poincar\'e equations
\[
\frac{d}{dt}\left( \fpd{l}{w}\right) = \ad_{w}^*\fpd{l}{w},
\]
(where $\ad^*$ is the adjoint action of $\g$ on its dual), or
equivalently if $C^k_{ij}$ are the structure constants of $\g$
corresponding to the basis used to define the coordinates,
\begin{equation}\label{EPeqI}
\frac{d}{dt}\left( \fpd{l}{w^j}\right) = C^k_{ij}\fpd{l}{w^k}w^i.
\end{equation}
To obtain the corresponding solution $g(t)$ of the Euler-Lagrange
equations we need to solve in addition the equation
$g(t)^{-1}\dot g(t) = w(t)$.

The invariant inverse problem on a Lie group $G$ has therefore the
following equivalent reduced version:\ if $\Gamma$ is invariant, when
does there exist a Lagrangian $l\in\cinfty{\g}$ such that its
Euler-Poincar\'e equations (\ref{EPeqI}) are equivalent to the
equations ${\dot w}^i = \gamma^i$ for the reduced vector field
$\gamma$ on $\g$?  As we will show in Theorem~\ref{thm2} below, if
such a Lagrangian $l$ exists for $\gamma$, the original vector field
$\Gamma$ will be the Euler-Lagrange field for some invariant
Lagrangian $L$.  The advantage of such an approach is that the
Lagrangian being sought is simply a function of the coordinates $w^i$
on the Lie algebra $\g=T_eG$, rather than a function of the
coordinates $(q^i,{\dot q}^i)$ on $TG$ satisfying invariance
conditions.  The solution to this existence problem will be given in
part by a set of reduced Helmholtz conditions for $\gamma$, involving
a multiplier matrix $(k_{ij})$ which, in the end, is the Hessian of
the function $l$ we want to find.  In addition, cohomological
conditions will again make their appearance here.  It turns out, as we will
establish in Theorem~\ref{thm4} below, that one of the functions of
the reduced Helmholtz conditions is to ensure that certain cochains
are cocycles, and determine cohomology classes in the cohomology of
$\g$.  What is not resolved by the Helmholtz conditions is whether
these cocycles can be made into coboundaries; that they can is the
additional requirement for the existence of a Lagrangian.

The two approaches, leading respectively to Theorem~\ref{thm1} and
Theorem~\ref{thm4}, involve classes in the cohomology of $\g$ which,
while differently derived, are the same.  Nevertheless there is a
subtle difference between the two forms of the inverse problem,
which it is worth pointing out.  The procedure described in
Theorem~\ref{thm1} and \cite{Marmo2} associates with a certain set
of Lagrangians a pair of cohomology classes, whose vanishing is the
condition for there to be a Lagrangian in the set which is
invariant.  The procedure described in Theorem~\ref{thm4} in effect
associates with a certain set of invariant functions a pair of
cohomology classes, whose vanishing is the condition for there to be
an invariant function in the set which is a Lagrangian.

The geometrical framework that we will use is based on a
reformulation of the Euler-Lagrange equations and of the Helmholtz
conditions in terms of a suitable adapted frame.  The requisite
background material is given in Section~2.  The solution of the
invariant inverse problem using invariant multipliers in the
Helmholtz conditions on $TG$ is presented in Theorem~\ref{thm1} in
Section~3.  The Euler-Poincar\'{e} equations are derived in
Section~4, and the reduced Helmholtz conditions in Section~5.
Section~6 is devoted to the proof of Theorem~\ref{thm4}, which is
the solution of the invariant inverse problem using the reduced
Helmholtz conditions, and is the main result of the paper. Next, we
investigate the geometric structure of the reduced space. In
Section~7 we show that Equation (\ref{EPeqI}) is a particular
example of a so-called Lagrangian system on a Lie algebroid, where
the Lie algebroid at hand is related in a natural way to the Lie
algebra $\g$ of the Lie group $G$.  We make the link between the
current set-up and Mart\'\i nez's approach \cite{Ed} to Lagrangian
systems on Lie algebroids. This will result in a
coordinate-independent reformulation of the reduced Helmholtz
conditions and of the cohomology conditions. The paper ends with
some examples and some suggestions for future work.

Although the paper focusses entirely on left-invariant Lagrangians,
it can easily be adjusted to the right-invariant case.

\section{Calculus along the tangent bundle projection}

One can find in the literature several reformulations of the
Helmholtz conditions that are independent of the choice of
coordinates on the manifold $M$:\ see for example
\cite{Crampin,OldMarmo}.  We will follow closely the one given in
\cite{Towards,MCS}, which is based on a calculus of tensor fields
along the tangent bundle projection $\tau: TM\to M$.  By a vector
field along $\tau$ we mean a section of the pullback bundle
$\tau^*TM \to TM$, and likewise for tensor fields.  A section of
$\tau^*TM \to TM$ can be interpreted as a map $X: TM \to TM$ with
the property that $\tau\circ X = \tau$, and can be expressed in
terms of local coordinates as
\[
X=X^i(q,\dot q)\fpd{}{q^i} \in\vectorfields\tau.
\]
There is a 1-1 correspondence between vector fields along $\tau$ and
vertical vector fields on $TM$. This correspondence is made explicit by the
so-called vertical lift $\vlift{X}$ of $X$, given by
\[
\vlift X= X^i\fpd{}{{\dot q}^i}.
\]
Any vector field on $M$ induces a vector field along $\tau$ in an
obvious way. If $X=X^i(q)\partial/\partial q^i$ is a vector field on $M$, its
complete lift $\clift{X}$ is the following vector field on $TM$:
\[
\clift X= X^i \fpd{}{q^i} + \fpd{X^i}{q^j}{\dot q}^j\fpd{}{{\dot q}^i}.
\]
Here are some convenient formulae for the brackets of
complete and vertical lifts:
\[
[\clift X,\clift Y] = [X,Y\clift], \qquad [\clift X,\vlift Y] =
[X,Y\vlift] \qquad\mbox{and}\qquad [\vlift X,\vlift Y] = 0.
\]
Here $X$ and $Y$  are vector fields on $M$ throughout.

The vertical and complete lifts require no additional machinery for
their definitions. However, if we have a second-order differential
equation field, or dynamical vector field, $\Gamma$ at our disposal,
say
\[
\Gamma={\dot q}^i\fpd{}{q^i}+f^i\fpd{}{{\dot q}^i}
\]
(representing the second-order equations $\ddot{q}^i=f^i$), we can use
it to define the so-called horizontal lift of a vector field along
$\tau$.  The horizontal lift $\hlift{X}$ of $X\in\vectorfields\tau$ is
\[
\hlift X= X^i \left(\fpd{}{q^i} - \Gamma^j_i \fpd{}{{\dot q}^j}\right), \qquad
\Gamma^j_i =-\onehalf \fpd{f^j}{{\dot q}^i}.
\]
Any vector field $Z$ on $TM$ can be decomposed into a horizontal and
vertical component:\
$Z=\hlift X + \vlift Y$, for  $X,Y \in \vectorfields\tau$.
In case $X$ is induced by a vector field on $M$, the three lifts are
related as follows:
\[
\hlift X =\onehalf (\clift X - [\Gamma,\vlift X]).
\]
Another useful fact, which it is easy to establish by a coordinate
calculation, is that $[\Gamma,\clift{X}]$ is always
vertical.

The Lie brackets of the dynamics $\Gamma$ with horizontal and
vertical vector fields define important objects for the calculus
along $\tau$. It can be shown that the horizontal and vertical
components of these brackets take the form
\[
[\Gamma,\vlift X] = -\hlift X + (\nabla X\vlift)
\qquad\mbox{and}\qquad [\Gamma,\hlift X]= (\nabla X\hlift) +
(\Phi(X)\vlift).
\]
The operator $\Phi$ is a type (1,1) tensor field along $\tau$ and is
called the Jacobi endomorphism.  The other operator, $\nabla$, acts as
a derivative on $\vectorfields\tau$, in the sense that for
$f\in\cinfty{TM}$ and $X\in\vectorfields\tau$, $\nabla(fX) = f\nabla X
+ \Gamma(f)X$.  It is therefore called the dynamical covariant
derivative.  Finally we will need the vertical derivative $\DV{X}$
associated with any $X\in\vectorfields\tau$.  This acts on vector
fields along $\tau$, but is completely determined by its actions on
vector fields $Y$ on $M$ and on functions $f$ on $TM$ by the formulae
$\DV{X} Y =0$ and $\DV{X} f = \vlift{X}(f)$.

In the framework of the calculus along the tangent bundle
projection the multiplier matrix is regarded as an operator
$g:\vectorfields\tau \times \vectorfields\tau\to \cinfty{TM}$, that
is as a type (0,2) tensor field along $\tau$, with local expression
$g=g_{ij}(q,\dot q) dq^i \otimes dq^j$. The actions of both the
dynamical covariant derivative and the vertical derivative can easily be
extended to (0,2) tensor fields along $\tau$:\ by definition, for
$X,Y,Z\in\vectorfields\tau$
\[
\nabla g (X,Y) = \Gamma(g(X,Y)) - g(\nabla X, Y)-g(X,\nabla Y)
\]
and
\[
\DV X g (Y,Z) = \vlift X(g(Y,Z)) - g(\DV{X}Y,Z)-g(Y,\DV{X}Z).
\]

The inverse problem can now be rephrased as the search for a
type (0,2) tensor field $g$ along $\tau$ which is non-singular and
satisfies for all
$X,Y,Z\in\vectorfields\tau$ the conditions
\begin{eqnarray*}
&&g(X,Y)=g(Y,X),\\&& \nabla g=0,\\
&&g(\Phi(X),Y)=g(X,\Phi(Y)),\\ && \DV X g (Y,Z) = \DV Y g (X,Z).
\end{eqnarray*}
These are the Helmholtz conditions in coordinate-independent form.

It will also be desirable to have a coordinate-independent version of
the Euler-Lagrange equations.  It is easy to see that the
Euler-Lagrange field $\Gamma$ of a regular Lagrangian $L$ is uniquely
determined by the fact that it is a second-order differential
equation field and satisfies
\[
\Gamma(\vlift X(L)) - \clift X(L)=0
\]
for every vector field $X$ on $M$. In particular,
if $\{X_i\}$ is a basis of vector fields on $M$ then $\{\clift
X_i, \vlift X_i\}$ is an induced basis for vector fields on $TM$, and
the  following set of equations is equivalent to the Euler-Lagrange
equations:
\[
\Gamma(\vlift X_i(L)) - \clift X_i(L)=0.
\]

We now consider the effect of a diffeomorphism of $M$ on the
Euler-Lagrange equations. Let $\varphi$ be a diffeomorphism of $M$
and $T\varphi$ the induced diffeomorphism of $TM$.  For any
$X\in\vectorfields{M}$, $T(T\varphi)(\clift{X})=\clift{(T\varphi
X)}$ and $T(T\varphi)(\vlift{X})=\vlift{(T\varphi X)}$ (these are of
course the counterparts of the bracket relations quoted earlier).
Moreover, if $\Gamma\in\vectorfields{TM}$ is a second-order
differential equation field so is $T(T\varphi)\Gamma$ (this is the
counterpart of the fact, stated earlier, that $[\Gamma,\clift{Z}]$
is always vertical).

Let $L$ be a regular Lagrangian with Euler-Lagrange field $\Gamma$.
Then $T\varphi^*L$ is a regular Lagrangian; we claim that its
Euler-Lagrange field is $T(T\varphi)^{-1}\Gamma$.  The proof
goes as follows.  For any function $f$, vector field $X$ and
diffeomorphism $\varphi$, $X(\varphi^*f)=\varphi^*((T\varphi X)(f))$.
We know that $\Gamma$ is uniquely determined by the Euler-Lagrange
equations $\Gamma(\vlift{X}(L))=\clift{X}(L)$ for all
$X\in\vectorfields{M}$.  Now
\begin{eqnarray*}
\clift{X}(T\varphi^*L)&=&
T\varphi^*\left(\clift{(T\varphi X)}(L)\right)=
T\varphi^*\left(\Gamma(\vlift{(T\varphi X)}(L))\right)\\
&=&
T(T\varphi)^{-1}\Gamma\left(T\varphi^*(\vlift{(T\varphi X)}(L)\right)=
T(T\varphi)^{-1}\Gamma\left(\vlift{X}(T\varphi^*L)\right).
\end{eqnarray*}

If $L$ is regular and $T\varphi^*L=L$ then
$T(T\varphi)\Gamma=\Gamma$. But although the Lagrangian uniquely
determines the Euler-Lagrange equations, it is not in general true
that the Euler-Lagrange equations uniquely determine the Lagrangian,
so if $T(T\varphi)\Gamma=\Gamma$ all we can conclude is that
$T\varphi^*L$ is a Lagrangian for $\Gamma$; if different from $L$ it
may be called an alternative Lagrangian.  That genuinely alternative
Lagrangians (Lagrangians not differing by a total derivative) can
exist even in the most familiar circumstances is well-known:\ the
free particle is the most obvious example, and lest that look too
suspiciously special we could mention also motion in a spherically
symmetric potential in Euclidean 3-space~\cite{Henn}.

\section{The invariant inverse problem}

For the remainder of the paper the configuration manifold $M$ will
be a connected Lie group $G$. We will use $\lambda_g$ and $\rho_g$
to denote left and right multiplication by $g\in G$.
Both maps can be extended to actions $T\lambda_g$ and $T\rho_g$ of
$G$ on $TG$.

We assume that we have a left-invariant second-order differential
equation field $\Gamma$ on $TG$:\ thus $T(T\lambda_g)\Gamma=\Gamma$
for all $g\in G$.  The question under discussion is whether $\Gamma$
admits an invariant regular Lagrangian, that is, whether there is a
function $L$ on $TG$ whose Hessian with respect to velocity
coordinates is non-singular and which satisfies $T\lambda_g^*L=L$
for all $g\in G$, such that $\Gamma$ is the Euler-Lagrange field of
$L$.  We can conclude from the analysis at the end of the last
section that the Euler-Lagrange field of an invariant regular
Lagrangian is invariant.  But if we start with an invariant
second-order differential equation field on the other hand, and it
admits a regular Lagrangian, then all we can conclude is that its
left translates are alternative Lagrangians, possibly different.

We now begin to develop the machinery we need for a deeper study of
the problem.

By left-translating a basis $\{E_i\}$ of the Lie
algebra $\g$ of $G$ we obtain a left-invariant basis $\{\hat{E}_i\}$ of
$\vectorfields G$. Similarly, $\{\tilde{E}_i\}$ will denote the
right-invariant basis of $\vectorfields{G}$ obtained via right
translation. These bases are related by
\begin{equation} \label{A}
{\hat E}_i(g)=A^j_i(g){\tilde E}_j(g),
\end{equation}
where $(A^j_i(g))$ is the matrix representation of $\ad_g$; in
particular $A^j_i(e)=\delta^j_i$ (where $e$ is the identity of $G$).
We will identify the Lie algebra with the left-invariant vector
fields:\ then $[{\hat E}_i,{\hat E}_j] = C_{ij}^k {\hat E}_k$ where
the $C^k_{ij}$ are the structure constants of $\g$, and $[{\tilde
E}_i,{\tilde E}_j] = -C_{ij}^k {\tilde E}_k$.  (This is the convention
in \cite{MR}, for example.)

In the following, a vector $v_g$ in $T_g G$ will have coordinates
$(w^i)$ with respect to $\{{\hat E}_i\}$, so that $v_g=w^i{\hat
E}_i(g)$.  Then $(w^i)$ are exactly the coordinates of the Lie algebra
element $w=T\lambda_{g^{-1}}v_g$ with respect to the basis $\{E_i\}$
of $\g$.

The following property is true for any action of a connected Lie group
on a manifold:\ a tensor field is invariant under an action if and
only if its Lie derivative by every fundamental vector field
vanishes.  When the manifold is a Lie group and the action is left
multiplication, the fundamental vector fields are the right-invariant
vector fields, for which $\{ {\tilde E}_i \}$ is a basis.  A function
$f$ on $G$ is left-invariant if and only if ${\tilde E}_i (f)=0$ for
all $i$, and a vector field $X$ on $G$ is left-invariant if and only if
$[{\tilde E}_i,X]=0$.  In particular, for the left-invariant ${\hat
E}_j$, $[{\tilde E}_i,{\hat E}_j]=0$.  In view of the bracket
relations in the two bases it follows that
\begin{equation}\label{conseq}
{\tilde E}_i(A^k_j) + A^l_jC^k_{li}=0 \qquad\mbox{and}\qquad
A^k_iA^l_jC^m_{kl}=A^m_n C^n_{ij}.
\end{equation}

The Lagrangian $L$ and the dynamical vector field $\Gamma$ both live
on the tangent manifold $TG$.  To characterize their invariance we
need to know the infinitesimal generators of the induced action
$T\lambda_g$ of $G$ on $TG$.  Given that the flow of a complete lift
is tangent to the flow of the underlying vector field, it is easy to
see that the infinitesimal generators of $T\lambda_g$ are exactly the
complete lifts $\{\clift{\tilde E}_i\}$ of the infinitesimal
generators of the action $\lambda_g$ of $G$ on $G$.  So a function
${F}\in\cinfty{TG}$ is left-invariant if and only if $\clift{\tilde
E}_i ({F})=0$, and a vector field $Z \in\vectorfields{TG}$ is
left-invariant if and only if $[\clift{{\tilde E}_i},Z]=0$. Note that
$\clift{\hat{E}_i}$ and $\vlift{\hat{E}_i}$ are invariant vector
fields, by virtue of the bracket relations for complete and vertical
lifts given earlier. The functions $w^i$ are also invariant; they are
linear fibre coordinates on $TG$, and satisfy
$\vlift{\hat{E}_j}(w^i)=\delta^i_j$.

The following observations will be important.  First, if a function
$f$ satisfies $\vlift{\hat{E}_i}(f)=0$ for all $i$ the $f$ is (the
pull-back to $TG$ of) a function on $G$.  Second, if
$\vlift{\hat{E}_i}(f)=f_i$ is a function on $G$ for all $i$ then
$f-f_iw^i$ is a function on $G$.

Recall that we interpret the Hessian of a Lagrangian $L$ as a type
(0,2) tensor field $g$ along the tangent bundle projection $\tau: TG
\to G$. If $L$ is invariant then the coefficients
$K_{ij}=\vlift{{\hat E}_i}\vlift{{\hat E}_j}(L)=g({\hat E}_i,{\hat
E}_j)$ will also be invariant functions.  Now when we use the
Helmholtz condition approach to the inverse problem, if we are
interested only in invariant Lagrangians we will certainly need to
add to the Helmholtz conditions the extra condition that the
multiplier $g$ should be invariant.  As we pointed out earlier, if
we start from an invariant second-order field, it is often possible
to find non-invariant Lagrangians. Examples of this behaviour can be
found in the papers \cite{Gerard4dim,Thompson2,Thompson1} for the
case of the canonical connection.  The reason is that in these
examples the extra condition about the invariance of the multiplier
is usually not imposed on the problem.  However, as we pointed out
before and will shortly explain in more detail, the invariance of
the multiplier, while necessary for the existence of an invariant
Lagrangian, is not sufficient.

The invariance of $g$ can be defined in a coordinate-independent way
as follows.  We first define a vector field $X$ along $\tau$ to be
left-invariant if its vertical lift $\vlift X$ is left-invariant.
Since $\{{\hat E}_i\}$ is a basis for $\vectorfields G$, it serves
also as a basis for vector fields along $\tau$.  Then a vector field
along $\tau$, $X=\Xi^i {\hat E}_i$, is invariant if $\clift{\tilde
E_j} (\Xi^i)=0$, or if its coefficients $\Xi^i\in\cinfty{TG}$ are
invariant functions.  We will say that a type (0,2) tensor field $g$
along $\tau$ is invariant if $g(X,Y)$ is an invariant function for all
invariant vector fields $X$ and $Y$ along $\tau$. It is easy to
verify that this holds if and only if the coefficients of $g$ with
respect to $\{{\hat E}_i\}$ are invariant.

We now state and prove a theorem which shows what requirements in
addition to the Helmholtz conditions and the invariance of the
multiplier are necessary and sufficient for the existence of an
invariant Lagrangian. Let us call a type (0,2) tensor field $g$
along $\tau$ which satisfies the Helmholtz conditions for an
invariant second-order differential equation field $\Gamma$ and is
invariant an invariant multiplier for $\Gamma$.

\begin{thm}\label{thm1}
An invariant multiplier for an invariant second-order differential
equation field $\Gamma$ determines a cohomology class in $H^1(\g)$ and
one in $H^2(\g)$.  The field $\Gamma$ is derivable from an invariant
Lagrangian if and only if the corresponding cohomology classes vanish.
\end{thm}

\begin{proof}
Suppose that $g$ is an invariant multiplier. We set $K_{ij}=g({\hat
E}_i,{\hat E}_j)$.  By the very fact that we have a solution of the
Helmholtz conditions we know that there is a regular Lagrangian for
$\Gamma$, say $L$, such that $K_{ij}=\vlift{{\hat E}_i}\vlift{{\hat
E}_j}(L)$. Now $L$ need not be invariant; but from the invariance of
the $K_{ij}$ we have
\[
0=\clift{\tilde{E}_k}(K_{ij})=
\vlift{\hat{E}_i}\vlift{\hat{E}_j}(\clift{\tilde{E}_k}(L))=0,
\]
whence $\clift{\tilde{E}_k}(L)=a_{kl}w^l+b_k$ for certain functions
$a_{kl}$ and $b_k$ on $G$. Since $L$ is known to be a Lagrangian and
$\Gamma$ is invariant,
\begin{eqnarray*}
0&=&\clift{\tilde{E}_i}\left(\Gamma(\vlift{\hat{E}_j}(L))-
\clift{\hat{E}_j}(L)\right)
=\Gamma(a_{ij})-\clift{\hat{E}_j}(a_{ik}w^k+b_i)\\
&=&w^k\left(\hat{E}_k(a_{ij})-\hat{E}_j(a_{ik})-a_{il}C^l_{jk}\right)
-\hat{E}_j(b_i).
\end{eqnarray*}
We can set to zero the coefficient of $w^k$ and the remaining term
separately (both are functions on $G$).  From the second we see that
$b_i$ is constant.  From the first,
\[
\hat{E}_k(a_{ij})-\hat{E}_j(a_{ik})-a_{il}C^l_{jk}=0.
\]
Let $\vartheta^i$ be the 1-forms on $G$ dual to the $\hat{E}_i$ (so
that $\vartheta=\vartheta^iE_i$ is the Maurer-Cartan form, not that it
matters); then for each $i$ the 1-form $a_{ij}\vartheta^j$ is closed,
from which it follows that $a_{ij}=\hat{E}_j(f_i)$ for
some functions $f_i$ on $G$.  We have
\begin{equation}\label{ejfi}
\clift{\tilde{E}_i}(L)=w^j\hat{E}_j(f_i)+b_i;
\end{equation}
note that this is of the form total derivative plus constant.
Next,
\[
0=\clift{\tilde{E}_i}\clift{\tilde{E}_j}(L)
-\clift{\tilde{E}_j}\clift{\tilde{E}_i}(L)+C_{ij}^k\clift{\tilde{E}_k}(L)
=w^k\hat{E}_k\left(\tilde{E}_i(f_j)-\tilde{E}_j(f_i)
+C_{ij}^lf_l\right)+C_{ij}^kb_k,
\]
from which it follows that
\begin{equation}\label{alpha}
\tilde{E}_i(f_j)-\tilde{E}_j(f_i)+C_{ij}^lf_l=\alpha_{ij}
\end{equation}
is constant, and $C_{ij}^kb_k=0$.
Now we can regard the $b_i$ as the coefficients, with respect to the
basis of $\g^*$ dual to the basis of $\g$ with which we are working,
of a linear map $b:\g\to\R$, so that $b(\xi)=b_i\xi^i$.
Similarly, the $\alpha_{ij}$ are the coefficients of an alternating
bilinear map $\alpha:\g\times\g\to\R$, so that
$\alpha(\xi,\eta)=\alpha_{ij}\xi^i\eta^j$.  We now show that, viewed from
the perspective of the cohomology of $\g$ with values in $\R$, $b$
and $\alpha$ are cocycles; that is, they satisfy the cocycle conditions
\[
b(\{\xi,\eta\})=0 \qquad\mbox{and}\qquad
\alpha(\xi,\{\eta,\zeta\})+\mu(\eta,\{\zeta,\xi\})+\mu(\zeta,\{\xi,\eta\})=0
\]
($\{\cdot,\cdot\}$ is the Lie algebra bracket); or in terms of the
structure constants,
\[
b_k C^k_{ij}=0\qquad\mbox{and}\qquad
\alpha_{il}C^l_{jk}+\alpha_{jl}C^l_{ki}+\alpha_{kl}C^l_{ij}=0.
\]
Indeed, we have just seen that $b_kC_{ij}^k=0$.  Operating with
$\tilde{E}_k$ again on Equation~(\ref{alpha}) and taking the cyclic
sum we see that $\alpha_{ij}$ is a cocycle too.  Moreover, $f_i$ is
determined only up to the addition of a constant; and the addition of
a constant leaves $b$ unchanged and changes $\alpha$ by a coboundary.

If $\alpha_{ij}$ and $b_i$ are both cohomologous to zero, then $b_i=0$,
and by choice of additive constants we can assume that
$\tilde{E}_i(f_j)-\tilde{E}_j(f_i)+C_{ij}^lf_l=0$. But then
$f_i=\tilde{E}_i(f)$ for some function $f$ on $G$. But then
$L-w^j\hat{E}_j(f)=L-\dot{f}$ is invariant, and of course has
$\Gamma$ as its Euler-Lagrange field and has the same Hessian as $L$.
\end{proof}

In \cite{Marmo2} the authors restrict their attention to Lagrangians
satisfying just $\clift{\tilde{E}_i}(L)=w^j\hat{E}_j(f_i)$, that is,
to Lagrangians which change only by addition of a total derivative
under the action of $G$; they call such Lagrangians quasi-invariant,
and appeal to physics to justify this choice.  From a purely
mathematical point of view such a restriction is unnecessary, and
the more general situation is easily analysed, as we have seen.  One
possible interpretation of the significance of the element of
$H^1(\g)$ is this:\ it is not difficult to see that
$b_i=-\clift{\tilde{E}_i}(\E)$, where $\E$ is the energy of $L$; so
$b_i=0$ is the condition for the energy to be invariant (even though
$L$ itself might not be).  We will have more to say about the
significance of $b_i$ later.

\section{The Euler-Poincar\'e equations} \label{SectionAd}

We now turn to the reduction of $\Gamma$ to the Lie algebra $\g$.

Left-invariant functions on $TG$ are in 1-1
correspondence with functions on the Lie algebra:\ on the one hand
restriction of any function on $TG$ to $T_eG$ determines a function
on $T_eG=\g$; on the other hand, any function on $\g=T_eG$ can be
extended to a left-invariant function on the whole of $TG$ by
requiring it to be constant along each orbit of the action.  From
now on we will use the following convention:\ capital letters such
as $F$ stand for left-invariant functions, vector fields, etc.\ on
$TG$; the corresponding small letters such as $f$ stand for their
restrictions to $T_eG=\g$.

A vector field $Z={\Xi}^j \clift{{\hat E}_j} + {F}^j \vlift{{\hat
E}_j} \in\vectorfields{TG}$ is left-invariant if and only if
$[\clift{{\tilde E}_i},Z]=0$, that is, if and only if
$\clift{{\tilde E}_i}({\Xi}^j)=0$ and $\clift{{\tilde
E}_i}({F}^j)=0$.  Thus $Z$ is invariant if and only if its
components ${\Xi}^j$ and ${F}^j$ are all invariant functions.  We
can therefore identify them with functions $\xi^j$ and $f^j$ on the
Lie algebra.  Note that $f^j\vlift{\hat{E}_j}|_e$ can be identified
with a vector field on $T_eG$, since it is vertical; the same is not
true for $\xi^j\clift{\hat{E}_j}|_e$, however:\ it is defined on
$T_eG$, but as a vector field it is transverse to it.

A set $\{\xi^j\}$ of $n=\dim\g$ functions on $\g$ can be interpreted
in two equivalent ways.  First, the elements of the set could be
viewed as the coefficients of a $\cinfty{\g,\g}$-map, namely the map
$\xi: w \mapsto \xi^i(w)E_i$.  A second interpretation is to view them
as the components of a vector field $\bar{\xi}$ on $\g$, where
$\bar \xi=\xi^j\partial/\partial w^j$.
This equivalence of interpretations is a manifestation of the fact
that the vector bundles $T\la \to \la$ and $\la\times\la \to \la$ are
isomorphic, so there is a 1-1 correspondence between their sections.

The two sets $\{\xi^j\}$ and $\{f^j\}$ together define a section of
the vector bundle $\g\times T\g\to\g$, or equivalently the bundle
$\g\times\g\times\g\to\g$.  We will adopt the following convention:\
an invariant vector field $Z={\Xi}^j \clift{{\hat E}_j} + {F}^j
\vlift{{\hat E}_j} \in\vectorfields{TG}$ reduces to the section
$z=(\xi,f)$ of $\g\times T\g\to\g$ where the first element $\xi=\xi^j
E_j$ is interpreted as a $\cinfty{\g,\g}$-map and the second
$f=f^j\partial/\partial w^j$ is a vector field on $\g$.  In
particular, for an invariant second-order field
\[
\Gamma = w^i\clift{{\hat E}_j} + \Gamma^j \vlift{{\hat E}_j}
\in\vectorfields{TG}
\]
the first invariance condition, $\clift{{\tilde E}_i}(w^j)=0$, is
trivially satisfied, so the only condition is $\clift{{\tilde
E}_i}(\Gamma^j)=0$.  Let $\Delta$ be the identity map in
$\cinfty{\g,\g}$; then $\Gamma$ reduces to the section
$(\Delta,\gamma)$ of $\g\times T\g\to\g$, where
\[
\gamma = \gamma^i\fpd{}{w^i}\in\vectorfields{\g}
\]
will be often called the reduced vector field on $\g$.

Let $L\in\cinfty{TG}$ be a left-invariant regular Lagrangian with
Euler-Lagrange field $\Gamma$.  We have shown in
Section~2 that this second-order differential equation field can be
characterized by the equations
\begin{equation} \label{Lagreq}
\Gamma(\vlift{\hat E}_i (L)) - \clift{\hat E}_i (L)=0.
\end{equation}
We have also shown that if $L$ is left-invariant then so also is
$\Gamma$.  We now compute its reduced vector field $\gamma$ on $\g$.

The Euler-Lagrange
equations (\ref{Lagreq}) are of the form
\[
w^k\clift{\hat E}_k\vlift{\hat E}_i(L) + \Gamma^k\vlift{\hat
E}_k\vlift{\hat E}_i(L) - \clift{\hat E}_i (L)=0.
\]
With the help of (\ref{conseq}), the relations between the complete
and vertical lifts of elements in the two bases is given by
\begin{equation} \label{clift}
\clift{\hat E}_i = A^j_i \clift{\tilde E}_j +
w^kC^j_{ki}\vlift{\hat E}_j \qquad \mbox{and}\qquad \vlift{\hat
E}_i=A^j_i\vlift{\tilde E}_j.
\end{equation}
As a consequence the first term in the Euler-Lagrange equations vanishes:
\[
w^k\clift{\hat E}_k\vlift{\hat E}_i(L)= w^k A^j_k \clift{\tilde E}_j
\vlift{\hat E}_i(L) = w^k A^j_k \vlift{\hat E}_i\clift{\tilde E}_j(L)
+ w^k A^j_k [\clift{\tilde E}_j, \vlift{\hat E}_i ](L) =0.
\]
On the other hand, for the last term we get
\[
\clift{\hat E}_i (L) = w^kC^j_{ki}\vlift{\hat E}_j(L).
\]
The Euler-Lagrange equations, adapted to the frame $\{{\hat E}_i\}$,
are therefore
\[
\Gamma^k\vlift{\hat E}_k\vlift{\hat E}_i (L) = w^kC^j_{ki}\vlift{\hat
E}_j(L).
\]
Notice that when $L$ is globally defined and smooth, and regular,
$\Gamma$ must vanish when $w^k=0$, that is, on the zero section of
$TG$.  So a necessary condition for an invariant second-order
differential equation $\Gamma$ to be derivable from a global regular
invariant Lagrangian (or even one smooth and regular in a
neighbourhood of the zero section) is that $\Gamma$ should vanish on
the zero section. In fact this is just the requirement that $b_i=0$
in Theorem~\ref{thm1}, since
\begin{equation}\label{bi}
b_i=\clift{\tilde{E}_i}(L)|_{w^k=0}=\Gamma_{w^k=0}(\vlift{\tilde{E}_i}(L)).
\end{equation}

Let $l\in\cinfty{\g}$ be the restriction of the left-invariant
Lagrangian $L\in \cinfty{TG}$ to the Lie algebra.  Then the
restriction of $\vlift{\hat{E}_k}(L) $ to $\g$ is $\partial l/\partial
w^k$, and so on.  The defining relation for the reduced vector field
$\gamma\in\vectorfields{\g}$ of $\Gamma$ is
therefore
\begin{equation}\label{EPeq}
\gamma \Big(\fpd{l}{w^l}\Big) = C^j_{ml}w^m\fpd{l}{w^j}.
\end{equation}
These are the so-called Euler-Poincar\'e equations \cite{MR}.

Evidently if $l$ is globally defined, smooth and regular on $\g$
then $\gamma$ must vanish at the origin (this is the counterpart of
the property of $\Gamma$ noted above).  So for a vector field
$\gamma$ on $\g$ to be derivable via the Euler-Poincar\'{e}
equations from a smooth and regular (reduced) Lagrangian it is
necessary that $\gamma(0)=0$.  We will come back to this point
later.

The Euler-Poincar\'e equations should be interpreted as differential
equations with solution $w(t)$ in the Lie algebra.  We have chosen the
coordinates $(w^i)$ in such a way that they are not only the
coordinates for $w=w^i E_i$ in $\g$, but also the fibre coordinates of
any translate $v_g =T\lambda_g w \in T_gG$.  To find the solution
$(g(t),{\dot g}(t))\in TG$ of the Euler-Lagrange equations that
corresponds to $w(t)$, one simply needs to integrate the equation
$g^{-1}(t)\dot g(t) =w(t)$.

So far as the inverse problem is concerned, we can use the foregoing
analysis to reduce the problem to one on $\g$, as set out in the
following theorem.

\begin{thm}\label{thm2}
Let $\Gamma$ be an invariant second-order differential equation
field on a Lie group $G$, and $\gamma$ the corresponding reduced
vector field on $\g$.  Then $\Gamma$ admits a regular invariant
Lagrangian $L$ on $TG$ if and only if $\gamma$ admits a regular
Lagrangian on~$\g$, in the sense that there is a smooth function $l$
whose Hessian is non-singular, such that $\gamma$ is the vector
field uniquely determined by the Euler-Poincar\'{e} equations of
$l$.
\end{thm}

\begin{proof}
Clearly, if $L$ is a regular invariant Lagrangian for $\Gamma$, its
restriction $l$ to $\g$ is a regular Lagrangian for $\gamma$.
Conversely, suppose that $l$ is a regular Lagrangian for $\gamma$ on
$\g$, and let $L$ be the unique invariant function on $TG$ which
agrees with $l$ on $T_eG=\g$. Consider the functions
\[
\varphi_i=\Gamma^k\vlift{\hat E}_k\vlift{\hat E}_i (L) -
w^kC^j_{ki}\vlift{\hat E}_j(L),
\]
where $\Gamma=w^k\clift{\hat{E}_k}+\Gamma^k\vlift{\hat{E}_k}$. We
showed earlier that $\Gamma^k$ is invariant, and so is $w^k$. Since
$\clift{\tilde{E}_i}$ commutes with $\vlift{\hat{E}_j}$, both
$\vlift{\hat E}_j(L)$ and $\vlift{\hat E}_k\vlift{\hat E}_i (L)$ are
invariant. So $\varphi_i$ is invariant. But the restriction of
$\varphi_i$ to $\g$ vanishes, by the Euler-Poincar\'{e} equations;
so $\varphi_i$ vanishes everywhere on $TG$. But as we showed
earlier, the vanishing of $\varphi_i$ is equivalent to the
Euler-Lagrange equations for $L$. Moreover, $L$ is regular since $l$
is. Thus $L$ is a regular invariant Lagrangian and $\Gamma$ is its
Euler-Lagrange field.
\end{proof}

\section{The reduced Helmholtz conditions}

In this section we will show that in the case of an invariant
Lagrangian, not only the Euler-Lagrange equations, but also the
Helmholtz conditions can be restated as conditions at the level of
the Lie algebra.

Recall that we interpret the Hessian of a Lagrangian as a
type (0,2) tensor field $g$ along the tangent bundle projection $\tau: TG
\to G$.  Due to the invariance of the Lagrangian the
coefficients $K_{ij}=\vlift{{\hat E}_i}\vlift{{\hat E}_j}(L)=g({\hat
E}_i,{\hat E}_j)$ will also be invariant functions. In what follows we
will denote the restrictions of these functions to $\g$ by $k_{ij}$.

Let us now evaluate the Helmholtz conditions, which we have stated in a
coordinate free way in the second section, in the basis $\{{\hat E}_i\}$.
The first conditions are simply
\begin{equation}\label{H1}
\det (K_{ij})\neq 0, \qquad  K_{ij}=K_{ji}.
\end{equation}
The Jacobi endomorphism and the dynamical derivative are determined
by the horizontal structure on $TG$. Since $[\Gamma,\vlift{\hat
E}_i]= - \clift{\hat E}_i + (w^jC^k_{ji}-\vlift{\hat
E}_i(\Gamma^k))\vlift{\hat E}_k $, it is easy to see that the
horizontal lift of ${\hat E}_i$ is
\[
\hlift{\hat E}_i = \clift{\hat E}_i + \onehalf
\left(-w^jC^k_{ji}+\vlift{\hat E}_i(\Gamma^k)\right)\vlift{\hat E}_k
= \clift{\hat E}_i -\Lambda_i^k \vlift{\hat E}_k
\]
say. Now both $\clift{\tilde E}_i(w^j)=0$ and $\clift{\tilde
E}_i(\Gamma^j)=0$, so all $\hlift{\tilde E}_i$ are invariant.  From
now on
\[
\lambda^k_i =
-\onehalf\left(\fpd{\gamma^k}{w^i}-w^jC^k_{ji}\right)
\]
denotes the restriction of the invariant function $\Lambda^k_i$ to
$\g$. It is easy to see that the horizontal lift of an invariant
vector field along $\tau$ is invariant, and vice versa.

We next consider the dynamical covariant derivative $\nabla$. We have
$[\Gamma,\vlift{\hat E}_i] = - \hlift{\hat E}_i +
\vlift{(\nabla\hat{E}_i)}$.
Now both $\Gamma$ and $\vlift{\hat E}_i$ are invariant, so by the
Jacobi identity $[\Gamma,\vlift{\hat E}_i]$ must be invariant also.
Since the horizontal part of the bracket, $-\hlift{\hat E}_i$, is
invariant, $\vlift{(\nabla {\hat E}_i)}$ and therefore $(\nabla {\hat
E}_i)$ must be invariant in turn. In general, if $X=X^i {\hat
E}_i\in\vectorfields\tau$ is invariant, then
$\nabla X = X^i \nabla{\hat E}_i + \Gamma(X^i){\hat E}_i$
is also invariant.  We may summarize this result by saying that $\nabla$
itself is invariant.  Furthermore, the coefficients of $\nabla$ with
respect to the invariant basis are invariant functions, which can be
reduced to functions on $\g$. In fact we can calculate $[\Gamma,\vlift{\hat
E}_i]$ explicitly, obtaining
$[\Gamma,\vlift{\hat E}_i] = - \hlift{\hat
E}_i+\Lambda^k_i\vlift{\hat E}_k$, so that
\[
\nabla {\hat E}_i = \onehalf(w^jC_{ji}^k -\vlift{{\hat
E}}_i(\Gamma^k)){\hat E}_k = \Lambda_i^k{\hat E}_k,
\]
and the coefficients are just the functions $\Lambda^k_i$ which we
know already to be invariant.

Given that $\clift{{\tilde E}_k}(K_{ij})=0$ the Helmholtz condition
$\nabla g=0$, when evaluated on the pair $({\hat E}_i,{\hat E}_j)$,
gives
\begin{equation}\label{H2}
\Gamma^k \vlift{\hat E}_k(K_{ij}) - K_{kj}\Lambda^k_i -
K_{ik}\Lambda^k_j=0.
\end{equation}

The components of the Jacobi endomorphism with respect to the
current basis can be calculated from $[\Gamma,\hlift{\hat E}_j]$.
One finds that
\begin{eqnarray*}
\Phi ({\hat E}_j)  &=& \left(
\onehalf\Gamma^i\vlift{\hat E}_i\vlift{\hat E}_j(\Gamma^l)
+\onehalf\Gamma^i C^l_{ij}
-\oneq\vlift{\hat E}_j(\Gamma^i) \vlift{\hat E}_i(\Gamma^l)\right. \\
&&\left.\mbox{}
-\threeq C^k_{ij}w^i \vlift{\hat E}_k (\Gamma^l)
+\oneq w^iC^l_{ik}\vlift{\hat E}_j(\Gamma^k)
-\oneq w^m w^n C^k_{mj}C^l_{nk}\right) {\hat E}_l
=\Phi_j^l {\hat E}_l.
\end{eqnarray*}
Again, the coefficients $\Phi_j^l$ are invariant functions, and
restrict to functions on $\g$ given by
\[\phi_j^l  =
\onehalf\gamma^i \frac{\partial^2\gamma^l}{\partial w^i \partial w^j}
+\onehalf\gamma^i C^l_{ij}
-\oneq\fpd{\gamma^i}{w^j}\fpd{\gamma^l}{w^i}
-\threeq C^k_{ij}w^i\fpd{\gamma^l}{w^k}
+\oneq w^iC^l_{ik}\fpd{\gamma^k}{w^j
}-\oneq w^m w^n C^k_{mj}C^l_{nk}.
\]
This somewhat uncouth-looking formula can be civilized by expressing
it in terms of the quantities
\[
\psi^i_j=\onehalf\left(\fpd{\gamma^i}{w^j}+C^i_{kj}w^k\right),
\]
when it becomes
\[
\phi^l_j=\gamma(\psi^l_j)-w^kC^i_{kj}\psi^l_i+w^kC_{ki}^l\psi^i_j
-\psi^k_j\psi^l_k.
\]
Again, for any invariant $X$, $\Phi(X)$ is an invariant vector field
along $\tau$.  The Helmholtz condition involving the Jacobi
endomorphism is simply
\begin{equation} \label{H3}
K_{ij}\Phi^i_k= K_{ik}\Phi^i_j.
\end{equation}
Finally, the $\DV{}$-condition is
\begin{equation} \label{H4}
\vlift{\hat E}_l(K_{ij}) = \vlift{\hat E}_i(K_{lj}).
\end{equation}

The conditions (\ref{H1}), (\ref{H2}), (\ref{H3}) and (\ref{H4}) are
all invariant; it is therefore enough to find a solution
$k_{ij}\in\cinfty{\g}$ of the restriction of these conditions to
$\g=T_eG$, which may be called the reduced Helmholtz conditions.  The
solution of the full conditions on $TG$ can then be found by left
translating the solution on $\g$.

For any $\gamma=\gamma^i\partial/\partial w^i\in\vectorfields{\g}$, we
call a matrix $(k_{ij})$ of functions on $\g$ a multiplier matrix for
$\gamma$ if it satisfies the reduced Helmholtz conditions
\begin{eqnarray*}
&& \det (k_{ij})\neq 0, \qquad  k_{ij}=k_{ji}, \\[2mm]
&& \gamma^k\fpd{k_{ij}}{w^k} - k_{kj}\lambda^k_i - k_{ik}\lambda^k_j=0, \\[2mm]
&& k_{ij}\phi^i_k= k_{ik}\phi^i_j,\\[2mm]
&& \fpd{k_{ij}}{w^l} = \fpd{k_{lj}}{w^i}.
\end{eqnarray*}
We have shown
\begin{thm}\label{thm3}
Suppose given an invariant second-order differential equation field
$\Gamma$, with reduced vector field $\gamma$.  Then there is an
invariant multiplier matrix $(K_{ij})$ for $\Gamma$ on $TG$ if and only
there is a multiplier matrix $(k_{ij})$ for $\gamma$ on $\g$.
\end{thm}

\section{The reduced inverse problem}

Theorem~\ref{thm2} shows that the problem of finding an invariant
regular Lagrangian for an invariant second-order differential equation
field on $TG$ can be reduced to that of finding a regular Lagrangian
for the reduced vector field on $\g$.  From Theorem~\ref{thm3} we can
infer that the existence of a multiplier matrix, that is, a solution
of the reduced Helmholtz conditions, for the reduced vector field on
$\g$ is a necessary condition for it to admit a Lagrangian.  However,
as we know, the relationship between Helmholtz conditions and
Lagrangians in the invariant inverse problem is a little more
complicated than is the case for the ordinary inverse problem.  While
the existence of a multiplier matrix on $\g$ is sufficient to
guarantee the existence of an invariant multiplier matrix on $TG$, the
existence of an invariant multiplier matrix on $TG$ is not sufficient
to guarantee the existence of an invariant Lagrangian on $TG$.  The
following theorem supplies in effect the extra conditions, working now
entirely in terms of reduced quantities on $\g$.

\begin{thm}\label{thm4}
A multiplier matrix for $\gamma\in\vectorfields{\g}$ determines a
cohomology class in $H^1(\g)$ and one in $H^2(\g)$.  The vector field
$\gamma$ is derivable from a Lagrangian if and only if the
corresponding cohomology classes vanish.
\end{thm}

\begin{proof}
Suppose the functions $k_{ij}$ on $\g$ satisfy the reduced Helmholtz
conditions, so that $(k_{ij})$ is a multiplier matrix.  From the last
of the Helmholtz conditions,
\[
\fpd{k_{ik}}{w^j} = \fpd{k_{ij}}{w^k},
\]
and the assumed symmetry of $k_{ij}$ in its indices, it follows that
there is a function $l$ on $\g$ such that
\[
k_{ij}=\spd{l}{w^i}{w^j};
\]
$l$ is determined up to the addition of a term linear in the $w^k$
(and the addition of a constant, but this we can ignore).  Then
\[
\vf{w^i}\left(\gamma\left(\fpd{l}{w^j}\right)-C^l_{kj}w^k\fpd{l}{w^l}\right)
=\gamma^k\fpd{k_{ij}}{w^k}+\fpd{\gamma^k}{w^i}k_{jk}-C^k_{ij}\fpd{l}{w^k}
-C^l_{kj}w^kk_{il}.
\]
Let us denote the term in brackets on the left-hand side (whose
vanishing is the Euler-Poincar\'{e} equations) by $V_j$.
Then the Helmholtz condition
\[
\gamma^k\fpd{k_{ij}}{w^k} - k_{kj}\lambda^k_i - k_{ik}\lambda^k_j=0
\]
is equivalent to
\[
\fpd{V_i}{w^j}+\fpd{V_j}{w^i}=0.
\]
It follows that
\[
\spd{V_i}{w^j}{w^k}=-\spd{V_j}{w^i}{w^k}=\spd{V_k}{w^i}{w^j}
=-\spd{V_i}{w^k}{w^j},
\]
whence
\[
\spd{V_i}{w^j}{w^k}=0.
\]
But this says that there are constants $\mu_{ij}$ and $\nu_i$, the
$\mu_{ij}$ being skew in their indices, such that
\begin{equation}\label{munu}
V_i=\gamma^k\spd{l}{w^i}{w^k}-C^l_{ki}w^k\fpd{l}{w^l}
=\mu_{ji}w^j+\nu_i.
\end{equation}
As before, we can regard the $\nu_i$ as the coefficients of a linear
map $\nu:\g\to\R$, and the $\mu_{ij}$ as the coefficients of an
alternating bilinear map $\mu:\g\times\g\to\R$.  We now show that $\nu$
and $\mu$  satisfy the cocycle conditions
$\nu_l C^l_{ij}=0$ and
$\mu_{il}C^l_{jk}+\mu_{jl}C^l_{ki}+\mu_{kl}C^l_{ij}=0$.
In fact these conditions hold by virtue of the Helmholtz conditions,
especially the condition $k_{il}\phi^l_j=k_{jl}\phi^l_i$.

Now $\mu_{ij}$ is half of the skew part of
\[
\fpd{V_j}{w^i}=
\gamma^k\fpd{k_{ij}}{w^k}+\fpd{\gamma^k}{w^i}k_{jk}-C^k_{ij}\fpd{l}{w^k}
-C^l_{kj}w^kk_{il},
\]
so that
\[
\mu_{ij}= \onehalf\left(\fpd{\gamma^l}{w^i}+C^l_{ki}w^k\right)k_{jl}
-\onehalf\left(\fpd{\gamma^l}{w^j}+C^l_{kj}w^k\right)k_{il}
-C^k_{ij}\fpd{l}{w^k}.
\]
Earlier, we set
\[
\onehalf\left(\fpd{\gamma^l}{w^i}+C^l_{ki}w^k\right)=\psi^l_i;
\]
we now put
\[
\chi_{ij}=\psi^l_ik_{jl}-\psi^l_jk_{il}=
\mu_{ij}+C_{ij}^k\fpd{l}{w^k}.
\]
We now consider
$\phi^l_ik_{jl}-\phi^l_jk_{il}$, where (as we showed earlier)
\[
\phi^l_i=\gamma(\psi^l_i)-w^kC^j_{ki}\psi^l_j+w^kC_{kj}^l\psi^j_i
-\psi^k_i\psi^l_k.
\]
We look first at the terms in $\phi^l_ik_{jl}-\phi^l_jk_{il}$ which
involve $\gamma(\psi^l_i)$:\ these are
\[
\gamma(\psi^l_i)k_{jl}-\gamma(\psi^l_j)k_{il}
=\gamma(\chi_{ij})-\psi^l_i\gamma(k_{jl})+\psi^l_j\gamma(k_{il}).
\]
We substitute for the  $\gamma(k_{il})$
terms from the appropriate Helmholtz condition, and find in the end that
\[
\phi^l_ik_{jl}-\phi^l_jk_{il}=
\gamma(\chi_{ij})+w^kC^l_{ki}\chi_{jl}-w^kC^l_{kj}\chi_{il}=0.
\]
Now since $\mu_{ij}$ is constant,
\[
\gamma(\chi_{ij})=C^l_{ij}\gamma\left(\fpd{l}{w^l}\right)
=C^l_{ij}\nu_l+
C^l_{ij}\left(\mu_{kl}+C_{kl}^m\fpd{l}{w^m}\right)w^k,
\]
so that
\begin{eqnarray*}
0&=&
\gamma(\chi_{ij})+(\chi_{il}C^l_{jk}+\chi_{jl}C^l_{ki})w^k\\
&=&C^l_{ij}\nu_l+C^l_{ij}\left(\mu_{kl}+C_{kl}^m\fpd{l}{w^m}\right)w^k
+(\chi_{il}C^l_{jk}+\chi_{jl}C^l_{ki})w^k\\
&=&C^l_{ij}\nu_l+(\mu_{kl}C^l_{ij}+\mu_{il}C^l_{jk}+\mu_{jl}C^l_{ki})w^k
+(C_{ij}^lC_{kl}^m+C_{jk}^lC_{il}^m+C_{ki}^lC_{jl}^m)w^k\fpd{l}{w^m}\\
&=&C^l_{ij}\nu_l+(\mu_{il}C^l_{jk}+\mu_{jl}C^l_{ki}+\mu_{kl}C^l_{ij})w^k.
\end{eqnarray*}
Since this expression
is affine in $w^k$ with constant coefficients, these coefficients
must vanish. Therefore both $C^l_{ij}\nu_l=0$ and
$\mu_{il}C^l_{jk}+\mu_{jl}C^l_{ki}+\mu_{kl}C^l_{ij}=0$, as required.

If we change $l$ to $l'=l+\theta_kw^k$, the corresponding change in
the cocycles is from $(\nu,\mu)$ to $(\nu',\mu')$ where
$\nu'_i=\nu_i$ and
$\mu'_{ij}=\mu_{ij}-\theta_kC^k_{ij}$, or $\nu'=\nu$
and $\mu'(\xi,\eta)=\mu(\xi,\eta)-\theta(\{\xi,\eta\})$. That is,
both components of $(\nu,\nu')$ and $(\mu,\mu')$ belong to the same
cohomology class, respectively. If the cohomology classes of $\nu$
and $\mu$ vanish then we can find $\theta$ such that
$l'=l+\theta_kw^k$ is a Lagrangian.
\end{proof}

By setting $w^i=0$ in Equation (\ref{munu}) we see that
\[
\nu_i=\gamma^k(0)\spd{l}{w^i}{w^k}(0).
\]
But as we pointed out earlier, it is a necessary condition for
$\gamma$ to be derivable from a Lagrangian $l$ on $\g$ that
$\gamma(0)=0$.  The significance of the vanishing of $\nu$ as a
condition for $\gamma$ to be derivable from a Lagrangian is clear.

We have derived two sets of conditions for the existence of an
invariant Lagrangian, each involving a pair of cohomology classes.
One would hope that the two pairs of cohomology classes are the same.
This is in fact the case, as we now show.

First we show that $b_i$ and $\nu_i$ are the same
constants. From Equation~(\ref{bi}) we have
\[
b_i=\Gamma_{w^k=0}(\vlift{\tilde{E}_i}(L))
=(\Gamma^j\vlift{\hat{E}_j}\vlift{\tilde{E}_i}(L))|_{w^k=0}.
\]
Since $b_i$ is constant it is enough to evaluate the right-hand side
at $e$; here the distinction between $\vlift{\tilde{E}_i}$ and
$\vlift{\hat {E}_i}$ disappears, and we obtain
\[
b_i=\gamma^k(0)\spd{l}{w^i}{w^k}(0)=\nu_i.
\]

To find the relationship between $\alpha_{ij}$ and $\mu_{ij}$ it
turns out to be convenient to work entirely in terms of the
right-invariant fields $\tilde{E}_i$; in the end we will evaluate
everything at $e$, using the constancy of the $\alpha_{ij}$, and
again we can take advantage of the fact that at $e$ the distinction
between $\tilde{E}_i$ and $\hat{E}_i$ disappears.  The relations
between the complete and vertical lifts of $\tilde{E}_i$ and
$\hat{E}_i$ given in Equation~(\ref{clift}) now come into play.  In
particular,
$\Gamma=w^jA^j_i\clift{\tilde{E}_i}+\Gamma^jA^j_i\vlift{\tilde{E}_i}$;
we set $A^i_jw^j=v^i$, $A^i_j\Gamma^j=\tilde{\Gamma}^i$.

It follows from Equation~(\ref{ejfi}) that
$\hat{E}_i(f_j)=\vlift{\hat{E}_i}(\clift{\tilde{E}_j}(L))$, whence
\[
\tilde{E}_i(f_j)=\vlift{\tilde{E}_i}(\clift{\tilde{E}_j}(L))
=\vlift{\tilde{E}_i}(\Gamma(\vlift{\tilde{E}_j}(L))).
\]
Using the expression above for $\Gamma$, and the evident fact that
$\vlift{\tilde{E}_i}(v^j)=\delta_i^j$, we find that
\begin{eqnarray*}
\lefteqn{\clift{\tilde{E}_i}(\vlift{\tilde{E}_j}(L))+
v^k\vlift{\tilde{E}_i}(\clift{\tilde{E}_k}(\vlift{\tilde{E}_j}(L)))}\\
&&\mbox{}+\vlift{\tilde{E}_i}(\tilde{\Gamma}^k)\vlift{\tilde{E}_k}(\vlift{\tilde{E}_j}(L))
+\tilde{\Gamma}^k\vlift{\tilde{E}_k}(\vlift{\tilde{E}_i}(\vlift{\tilde{E}_j}(L)))
-\vlift{\tilde{E}_i}(\clift{\tilde{E}_j}(L))=0,
\end{eqnarray*}
or (remembering that $[\tilde{E}_i,\tilde{E}_j]=-C_{ij}^k\tilde{E}_k$)
\begin{eqnarray*}
\lefteqn{\vlift{\tilde{E}_i}(\clift{\tilde{E}_j}(L))
-\vlift{\tilde{E}_j}(\clift{\tilde{E}_i}(L))+C_{ij}^k\vlift{\tilde{E}_k}(L)}\\
&&=\Gamma(\vlift{\tilde{E}_i}(\vlift{\tilde{E}_j}(L))
-C_{ik}^lv^k\vlift{\tilde{E}_l}(\vlift{\tilde{E}_j}(L))
+\vlift{\tilde{E}_i}(\tilde{\Gamma}^k)\vlift{\tilde{E}_k}(\vlift{\tilde{E}_j}(L)).
\end{eqnarray*}
On taking the skew part we find that
\begin{eqnarray*}
\lefteqn{\vlift{\tilde{E}_i}(\clift{\tilde{E}_j}(L))
-\vlift{\tilde{E}_j}(\clift{\tilde{E}_i}(L))
+C_{ij}^k\vlift{\tilde{E}_k}(L)}\\
&&=\onehalf\left(\vlift{\tilde{E}_i}(\tilde{\Gamma}^k)+C_{li}^kv^l\right)
\vlift{\tilde{E}_k}(\vlift{\tilde{E}_j}(L))
-\onehalf\left(\vlift{\tilde{E}_j}(\tilde{\Gamma}^k)+C_{lj}^lv^l\right)
\vlift{\tilde{E}_k}(\vlift{\tilde{E}_i}(L)).
\end{eqnarray*}
The left-hand side is $\alpha_{ij}+C^k_{ij}(\vlift{\tilde{E}_k}(L)-f_k)$.
At $e$, the right-hand side is
\[
\onehalf\left(\fpd{\gamma^k}{w^i}+C^k_{li}w^l\right)k_{jk}
-\onehalf\left(\fpd{\gamma^k}{w^j}+C^k_{lj}w^l\right)k_{ik}
=\psi_i^kk_{jk}-\psi^k_jk_{ik}=\chi_{ij}.
\]
Thus
\[
\alpha_{ij}=\chi_{ij}-C^k_{ij}(\vlift{\tilde{E}_k}(L)-f_k)|_e.
\]
Now let $l$ be the restriction of $L$ to $T_eG$. Of course $L$ is
not assumed to be invariant, so this differs from the association
between $l$ and $L$ given earlier; nevertheless, it is true that
\[
\spd{l}{w^i}{w^j}=k_{ij},
\]
where $(k_{ij})$ satisfies the reduced Helmholtz conditions.  So we
can write $\alpha_{ij}=\mu_{ij}+C^k_{ij}f_k(e)$.  It is apparent that
$\alpha_{ij}$ and $\mu_{ij}$ define the same cohomology class (they
differ by a coboundary).

\section{The Lie algebroid}

Our policy while working on $TG$ in earlier sections was to write
everything in terms of $G$-invariant quantities, that is, quantities
determined by their values on $T_eG=\g$.  This paves the way towards
expressing the whole theory in terms of $\g$, or more accurately in
terms of a vector bundle over $\g$, namely $\g\times T\g\to\g$.  We
can identify invariant vector fields on $TG$, via their restrictions
to $T_eG$, with sections of $\g\times T\g\to\g$, as we pointed out
earlier.  The bracket of two invariant vector fields remains
invariant, and so the bracket of vector fields on $TG$ determines a
bracket of sections of $\g\times T\g\to\g$.  This is evidently
$\R$-bilinear and skew, and it satisfies the Jacobi identity by
construction.  We will now obtain an explicit formula for this
bracket, and deduce that it is a Lie algebroid bracket, i.e.\ a
bracket with the above properties that satisfies an appropriate
Leibniz rule when sections are being multiplied with functions on
the base manifold.

Let $\xi^i\clift{\hat{E}_i}+X^i\vlift{\hat{E}_i}$ and
$\eta^i\clift{\hat{E}_i}+Y^i\vlift{\hat{E}_i}$ be two invariant vector
fields, so that $\clift{\tilde{E}_j}(\xi^i)= \clift{\tilde{E}_j}(X^i)=
\clift{\tilde{E}_j}(\eta^i)= \clift{\tilde{E}_j}(Y^i)=0$.  These
invariance conditions, when expressed in terms of the vector fields of
the invariant basis, become for example
$\clift{\hat{E}_j}(\xi^i)=w^kC^l_{kj}\vlift{\hat{E}_l}(\xi^i)$, using
Equation~(\ref{clift}).  Thus
\[
[\xi^i\clift{\hat{E}_i},\eta^j\clift{\hat{E}_j}]=
\left(\xi^i\eta^jC^k_{ij}+\xi^iw^jC^l_{ji}\vlift{\hat{E}_l}(\eta^k)
-\eta^iw^jC^l_{ji}\vlift{\hat{E}_l}(\xi^k)\right)\clift{\hat{E}_k},
\]
while
\[
[\xi^i\clift{\hat{E}_i},Y^j\vlift{\hat{E}_j}]=
-Y^j\vlift{\hat{E}_j}(\xi^k)\clift{\hat{E}_k}
+\left(\xi^iY^jC_{ij}^k+\xi^iw^jC^l_{ji}\vlift{\hat{E}_l}(Y^k)\right)
\vlift{\hat{E}_k}.
\]
The bracket may be written as follows. We identify $\xi$, $\eta$ with
$\g$-valued functions on $\g$, $X$, $Y$ with vector fields on $\g$;
$\bar{\xi}$ is the vector field corresponding to $\xi$. We think of
$w^kC^i_{kj}$ as the components of a type (1,1) tensor field on
$\g$ which we denote by $\A$:\ thus
\[
\A=w^kC^i_{kj}\vf{w^i}\otimes dw^j.
\]
The Lie algebra bracket $\{\cdot,\cdot\}$ extends naturally to an
algebraic bracket on $\g$-valued functions on $\g$, so that
$\{\xi,\eta\}=\xi^j\eta^kC^i_{jk}E_i$.  Then
\begin{eqnarray*}
\lequiv(\xi,X),(\eta,Y)\requiv&=&
\left(\{\xi,\eta\}+\A(\bar{\xi})(\eta)-\A(\bar{\eta})(\xi)+X(\eta)-Y(\xi),\right.\\
&&\quad\left.
[\A(\bar{\xi}),Y]-[\A(\bar{\eta}),X]+\A(\overline{Y(\xi)})-\A(\overline{X(\eta)})
+[X,Y]\right).
\end{eqnarray*}

For any function $f$ on $\g$ we have
$\lequiv(\xi,X),f(\eta,Y)\requiv=f\lequiv(\xi,X),(\eta,Y)\requiv
+\rho(\xi,X)(f)(\eta,Y)$, as required, where the so-called anchor of
the Lie algebroid is given by
\[
\rho(\xi,X)=\A(\bar{\xi})+X\in\vectorfields\g.
\]
 Thus the
bracket $\lequiv\cdot,\cdot\requiv$ does indeed define a Lie
algebroid structure on $\g\times T\g\to\g$.

We denote by $e_i$ the section $(E_i,0)$ of $\g\times T\g\to\g$, and
$W_i$ the section $(0,\partial/\partial w^i)$; then $\{e_i,W_i\}$ is
a basis of sections, and we have
\[
\lequiv e_i,e_j\requiv=C_{ij}^ke_k,\quad
\lequiv e_i,W_j\requiv=C_{ij}^kW_k,\quad
\lequiv W_i,W_j\requiv=0.
\]
We denote by $\delta$ the induced exterior derivative operator on
sections of exterior powers of the dual of the algebroid, and by
$\{e^i,W^i\}$ the basis dual to $\{e_i,W_i\}$. Then for any function
$f$ on $\g$,
\[
\delta f=\fpd{f}{w^i}(w^kC^i_{kj}e^j+W^i),
\]
while
\[
\delta e^i=-\onehalf C^i_{jk}e^j\wedge e^k,\quad
\delta W^i=-C^i_{jk}e^j\wedge W^k.
\]
Using these formulae we can express the Euler-Poincar\'{e} equations
in terms of the Lie algebroid structure, as follows.  The vertical
endomorphism $S$ on the Lie algebroid is just
$S(\xi,X)=(0,\bar{\xi})$. The invariant Lagrangian is represented by
a function $l$ on $\g$.  We define, in analogy to the usual case, a
Cartan form $\theta$ and an energy function $\E$ by
\[
\theta=S(\delta l), \quad \E=\langle\bar\Delta,\delta l\rangle-l,
\]
where $\Delta=w^i e_i$ and $\bar\Delta=w^iW_i$. As we will show by
a direct calculation, provided that $l$ is regular (in the sense
that its Hessian is non-singular) the equation
\[
i_\Gamma \delta\theta=-\delta \E
\]
determines a unique section $\Gamma$, which is of second-order
differential equation type, so that it takes the form
$w^ie_i+\gamma^iW_i$, and the $\gamma^i$ satisfy the
Euler-Poincar\'{e} equations for $l$. In fact
\begin{eqnarray*}
\theta&=&\fpd{l}{w^i}e^i\\
\delta\theta&=&
\left(\spd{l}{w^j}{w^l}w^kC^l_{ki}-\onehalf\fpd{l}{w^k}C^k_{ij}\right)e^i\wedge e^j
-\spd{l}{w^i}{w^j}e^i\wedge W^j\\
\E&=&w^i\fpd{l}{w^i}-l\\
\delta \E&=&
w^l\spd{l}{w^i}{w^l}(w^kC^i_{kj}e^j+W^i).
\end{eqnarray*}
Let us write $\Gamma=\xi^ie_i+f^iW_i$.  Then the vanishing of the
$W_i$ component of $i_\Gamma \delta\theta+\delta \E$ gives
\[
(-\xi^j+w^j)\spd{l}{w^i}{w^j}=0,
\]
whence $\xi^i=w^i$ when $l$ is regular.  When this result is inserted
in $i_\Gamma \delta\theta+\delta \E$ several terms cancel, and the remaining
terms in the $e_i$ component reduce to
\[
\gamma^j\spd{l}{w^i}{w^j}-w^kC^j_{ki}\fpd{l}{w^j},
\]
as required.

The above derivation of the Euler-Poincar\'e equations was inspired by
Mart\'inez' framework \cite{Ed} for Lagrangian systems on a Lie
algebroid.  This framework is based on the so-called prolongation
algebroid of the underlying Lie algebroid.  It should be remarked that
although the underlying algebroid of the current system is just the
Lie algebra $\g$, the algebroid we have defined in this section does
not coincide with the prolongation algebroid of the Lie algebra.  The
prolongation algebroid can most easily be defined as follows.  Observe
that both components of a section of $\g\times T\g\to\g$ have a
natural bracket structure.  For the first, we have the natural
extension of the Lie algebra bracket to $\cinfty{\g,\g}$-functions,
and for the second component we have the Lie bracket of vector fields.
The easiest way to combine these two into one Lie bracket structure is
as follows:
\begin{equation}\label{LABP}
\LABP{(\xi,X),(\eta,Y)} = (\{\xi,\eta\} + X(\eta) - Y(\xi),[X,Y]).
\end{equation}
It is easy to check that this is a Lie algebroid whose anchor map
$\rho^1: \la\times T\la\to T\la$ is simply the projection on the
second component.  In our basis \[ \LABP{e_i,e_j} = C^k_{ij} e_k,
\qquad \LABP{e_i,W_j} = 0, \quad\quad \LABP{W_i,W_j} = 0.
\]
A short calculation reveals that indeed the expression
$i_\Gamma\delta^1\theta^1 =-\delta^1 \E$ leads again to the
Euler-Poincar\'e equations. It is, however, easy to guess the
relationship between the two algebroid structures:\ the section map
$(\xi,X)\mapsto (\xi,\A(\bar{\xi})+X)$
is an isomorphism of the first of the Lie algebroids with the second.
The more complicated structure of the first Lie algebroid (or of that
presentation of the common Lie algebroid if one regards isomorphic
algebroids as identical in principle) arises from our desire to work
always with invariant objects on $TG$ and objects on $\g$ derived
from them by restriction.

We conclude this section by showing that the restrictions of the
horizontal lift, the Jacobi endomorphism and the dynamical
derivative to $\g$ have a direct interpretation in the first Lie
algebroid.

At the level of the Lie algebra, $\g$-valued functions on $\g$
(sections of $\g\times\g\to \g$) play the role that the vector fields
along the tangent bundle projection played at the level of $TG$.  That
is, there is a well-defined vertical lift of such a function
$\xi=\xi^i(w)E_i$ to the section $\vlift\xi= \xi^iW_i$ of the Lie
algebroid.  Moreover, one can easily verify that each vector field
$\gamma$ on $\la$ defines a horizontal lift
\[
\hlift{\xi} = \xi^i(e_i - \lambda_i^j W_j)
\]
to sections of the Lie algebroid, or equivalently, a splitting of
the short exact sequence
\[
0 \to \{0\}\times T\la \to \la \times T\la \to \la\times\la \to 0
\]
of vector bundles over $\g$. Observe that the restriction to $\g$ of
the horizontal lift of an invariant vector field along $\tau$ is in
fact the horizontal lift of the restriction to $\g$ of that vector
field along~$\tau$.

The functions $\psi_i^j$ we have introduced in previous sections
have a nice interpretation in the algebroid set-up. The projection
by the anchor map of the horizontal algebroid section $\hlift{E}_i =
e_i -\lambda_i^j W_j$ to a vector field on $\g$ is exactly
\[
\rho(\hlift{E}_i) = (w^kC_{ki}^l - \lambda_i^l)\fpd{}{w^l} =
\psi_i^l \fpd{}{w^l}.
\]

As before, we can define a Jacobi endomorphism and a dynamical
derivative by considering the horizontal and vertical parts of the
brackets of the algebroid section $\Gamma = w^ie_i+\gamma^iW_i$:
\[
\LAB{\Gamma,\vlift\eta}= - \hlift\eta + (\nabla\eta\vlift), \qquad
\LAB{\Gamma,\hlift\eta} = (\nabla\eta\hlift) + (\Phi(\eta)\vlift),
\]
where $\nabla$ acts like a derivative in the sense that for
$f\in\cinfty\g$, $\nabla(f\eta) = f\nabla\eta + \gamma(f)\eta$,
and $\Phi$ is tensorial with coefficients $\phi^i_j$ as before. We
have $\nabla E_i =  \lambda_i^k E_k$,
and in fact in both cases the operators are nothing but the original operators
restricted to $\la=T_eG$.

We define the vertical derivative for $\xi\in\cinfty{\g,\g}$ as the
map $\DV{\xi}: \cinfty{\g,\g} \to \cinfty{\g,\g}$, determined by
$\DV{\xi}f= \bar\xi(f)$ for $f\in\cinfty{\g}$, $\DV{\xi}\eta=0$ for a
vector $\eta$ of $\g$ (in other words a constant element of
$\cinfty{\g,\g}$), and the obvious Leibniz rule for multiplication
by functions.

We can now restate the reduced Helmholtz conditions in a
coordinate-free form.  The multiplier $(k_{ij})$ is a
matrix of functions on $\g$, or equivalently a map $k:\cinfty{\g,\g}
\times \cinfty{\g,\g}\to\cinfty{\g}$, which satisfies the conditions
\begin{eqnarray*}
&&\det k\neq 0,\quad k(\xi,\eta)=k(\eta,\xi),\\
&& \nabla k=0,\\
&& k(\Phi(\eta),\zeta)=k(\eta,\Phi(\zeta)),\\
&& \DV \xi k(\eta,\zeta) = \DV \eta k(\xi,\zeta).
\end{eqnarray*}
for all $\xi,\eta,\zeta\in\cinfty{\g,\g}$.

Suppose that the Hessian of $l\in\cinfty{\g}$ is $k$. Then if
$\theta=(\partial l/\partial w^i) e^i$ as before, we can define a
2-form $\hlift{\delta}\theta$ on the algebroid by requiring it to
vanish whenever one of its arguments is a vertical section (i.e.\ it
is semi-basic) and by setting
\[
\hlift{\delta}\theta (\hlift{\xi},\hlift{\eta}) =
\rho(\hlift{\xi})(\theta(\hlift{\eta})) -
\rho(\hlift{\eta})(\theta(\hlift{\xi})) - \theta (\lequiv
\hlift{\xi},\hlift{\eta}\requiv ).
\]
Then $\hlift{\delta}\theta =\onehalf\mu_{ij} e^i\wedge e^j$, where
\begin{eqnarray*}
\mu_{ij} &=& \rho(\hlift{E}_i)(\theta(\hlift{E}_j)) -
\rho(\hlift{E}_j)(\theta(\hlift{E}_i))
- \theta (\lequiv \hlift{E}_i,\hlift{E}_j\requiv )\\
&=& \psi_i^l k_{lj} - \psi_j^l k_{li} - C_{ij}^k \fpd{l}{w^k}.
\end{eqnarray*}
These coefficients are exactly those we have encountered in the
previous section.  Recall from the proof of Theorem~4 that the reduced
Helmholtz conditions ensure that the $\mu_{ij}$ are constants and that
they form a cocycle.  A necessary condition for a Lagrangian to exist
is that $\mu_{ij}$ is a coboundary.  We can now re-express this
statement in terms of the Lie algebroid.  For a 2-form
$\mu=\hlift{\delta}\theta$ with constant coefficients we get that
$\delta \mu = -\onehalf \mu_{ij}C^i_{lk}e^l \wedge e^k \wedge e^j$.
Therefore, $\delta$-closure of the 2-form $\mu$ amounts to the cocycle
condition.  On the other hand, the condition that the $\mu_{ij}$ are
of the form $\alpha_kC^k_{ij}$ for some $\alpha_k$ is equivalent to
$\mu$ being exact.  Similarly, the reduced Helmholtz conditions ensure
that the semi-basic 1-form $\nu=
i_\Gamma\delta\theta+\delta\E-i_\Gamma\mu$ has constant coefficients
$\nu_i$ that form a cocycle, or that $\delta\nu=0$.  Theorem~4 states
that $\nu$ should vanish for a Lagrangian to exist.

\section{Examples and applications}

The method of reduced Helmholtz conditions really comes into its own
when one has to deal with any specific problem.  In practice ---
certainly, if the following examples are representative --- the
cohomological conditions do not play much of a role.  Where there is
no invariant Lagrangian this is because the reduced Helmholtz
conditions fail, often at the level of regularity.  Where one is
able to find a solution of the Helmholtz conditions one is usually
able to integrate it by hand, and check directly for which
integration constants the Euler-Poincar\'e equations are equivalent
to the equations associated with the vector field $\gamma$.

To save space, in the following examples we will write $k_{ijl}$ for
$\partial k_{ij}/\partial w^l$, and we will implicitly assume that the
conditions $k_{ij}=k_{ji}$, $k_{ijl}=k_{ilj}$ and so on are satisfied.

\subsection{The canonical connection on a Lie group}

The canonical connection on a Lie group is defined as a covariant
derivative operator by $\nabla_X Y = \onehalf[X,Y]$, where $X$ and
$Y$ are any two left-invariant vector fields on $G$.  As we
mentioned in the Introduction, the invariant inverse problem for the
canonical connection has been studied by Muzsnay in \cite{Zolt};
however, he uses methods different from ours.

The connection coefficients of the canonical connection with respect
to the left-invariant basis $\{{\hat E}_i\}$ of $\vectorfields{G}$ are
just $\frac{1}{2}C^i_{jk}$.  So the coefficients $\Gamma^i$ of the
corresponding second-order differential equation field (the geodesic
spray) are in this case $\Gamma^i = \onehalf C^i_{jk} w^j w^k =0$.
The reduced equations are therefore simply ${\dot w}^i=0$.  In fact
the geodesics through the identity of $G$ are just the 1-parameter
subgroups.

If a left-invariant Lagrangian $L$ exists, then
\begin{equation}\label{Zolt1}
 C^k_{ij}\vlift{\hat E}_k(L)w^i=0 \qquad \mbox{or}\qquad
 C^k_{ij}\fpd{l}{w^k}w^i=0.
\end{equation}
In view of relation (\ref{clift}), $L$ must also be right-invariant,
i.e.\ $\clift{{\hat E}_j}(L)=0$, and thus bi-invariant.  At the
level of the Lie algebra, this means that $l\in\cinfty{\g}$ will be
$\ad$-invariant.  This observation is in fact Proposition~2 in
\cite{Zolt}. Thus a Lagrangian is a function which is constant on
the adjoint orbits in $\g$ and whose Hessian is non-singular.

We will use our methods to investigate the invariant inverse problem
for the canonical connection.  The first observation is that in this
case the reduced Helmholtz condition $k_{il}\phi^l_j=k_{jl}\phi^l_i$
is a consequence of the other conditions. Since $\gamma^i=0$,
\begin{equation}\label{nabla}
\gamma^k\fpd{k_{ij}}{w^k} - k_{kj}\lambda^k_i - k_{ik}\lambda^k_j
=\onehalf w^l(k_{kj}C^k_{li}+k_{ik}C^k_{lj})=0.
\end{equation}
On the other hand, $\phi_j^l=-\oneq w^m w^n C^k_{mj}C^l_{nk}$. But
\[
w^m w^n C^k_{mj}C^l_{nk}k_{il}-w^m w^n C^k_{mi}C^l_{nk}k_{jl}= -w^m
w^n C^k_{mj}C^l_{ni}k_{kl}+w^m w^n C^k_{mi}C^l_{nj}k_{kl}=0,
\]
so the condition $k_{il}\phi^l_j=k_{jl}\phi^l_i$ holds by virtue of
condition (\ref{nabla}) and the symmetry of $k_{kl}$.

A second general remark concerns the cohomological conditions. In this
case Equation~(\ref{munu}) reads
\[
\mu_{ji}w^j+\nu_i=-C^l_{ki}w^k\fpd{l}{w^l},
\]
from which immediately $\nu_i=0$. Moreover
\[
\mu_{ij}=-C^k_{ij}\fpd{l}{w^k}(0),
\]
so $\mu_{ij}$ is a coboundary.  Thus the cohomological conditions are
automatically satisfied for the canonical connection.  Any function
$l$ whose Hessian satisfies the reduced Helmholtz conditions and is
such that $C^k_{ij}\partial l/\partial w^k(0)=0$ will be a Lagrangian;
in particular, if $l$ satisfies the reduced Helmholtz conditions and
we set
\[
l'=l-w^k\fpd{l}{w^k}(0)
\]
then $l'$ will be a Lagrangian.  So the inverse problem for the
canonical connection reduces essentially to the analysis of
condition (\ref{nabla}), in the form
$w^l(k_{kj}C^k_{li}+k_{ik}C^k_{lj})=0$, and the condition
$k_{ijk}=k_{ikj}$.  Where there is no Lagrangian this will often
become apparent by the fact that there is no non-singular $(k_{ij})$
satisfying the first of these conditions.

We will examine two specific situations, one in which there is no
Lagrangian, one in which there is one.

The first case is that of the Heisenberg algebra, which is a
3-dimensional algebra with the single non-trivial bracket relation
$\{E_1,E_3\}=E_2$. Condition (\ref{nabla}) amounts simply to
\[
k_{12}w^3=k_{22}w^3=k_{22}w^1=k_{32}w^1=-k_{12}w^1+k_{32}w^3=0.
\]
Evidently $k_{12}=k_{22}=k_{32}=0$, and there is no non-singular
$3\times3$ matrix $(k_{ij})$ satisfying the Helmholtz conditions.

For our second example we take the 4-dimensional Lie algebra with
bracket relations
\[
\{E_2,E_3\}=E_1,\qquad \{E_2,E_4\}=E_2,\qquad \{E_3,E_4\}=-E_3
\]
(this is the algebra $\mathrm{A}_{4,8}$ in the classification of
Patera et al.\ \cite{Pat}).  Condition (\ref{nabla}) says in this
case that the matrix
\[
\left[\begin{array}{rrrr}
k_{11}&k_{12}&k_{13}&k_{14}\\
k_{12}&k_{22}&k_{23}&k_{24}\\
k_{13}&k_{23}&k_{33}&k_{34}\\
k_{14}&k_{24}&k_{34}&k_{44}\end{array}\right]
\left[\begin{array}{cccc}
0&w^3&-w^2&0\\
0&w^4&0&-w^2\\
0&0&-w^4&w^3\\
0&0&0&0\end{array}\right]
\]
must be skew-symmetric. This leads to the following 7 independent
equations for the 10 unknowns $k_{ij}$ (with $i\leq j$):
\begin{eqnarray*}
k_{11}w^3+k_{12}w^4&=&0=k_{11}w^2+k_{13}w^4,\\
k_{12}w^3+k_{22}w^4&=&0=k_{13}w^2+k_{33}w^4,\\
k_{24}w^2+k_{34}w^3&=&0,\\
k_{22}w^2-k_{24}w^4&=&(k_{14}+k_{23})w^3\\
k_{33}w^3-k_{34}w^4&=&(k_{14}+k_{23})w^2.
\end{eqnarray*}
Evidently $k_{44}$ is unconstrained by these equations. It turns out
that $k_{24}=k_{34}=0$. The remaining unknowns can conveniently be
expressed in terms of $k_{11}$ and $k_{23}$.  If for convenience we
set $k_{11}=(w^4)^2F$ (for $w^4\neq0$), $k_{23}=G$ and $k_{44}=H$
then $(k_{ij})$ is
\[
\left[\begin{array}{cccc}
(w^4)^2F&-w^3w^4F&-w^2w^4F&w^2w^3F-G\\
-w^3w^4F&(w^3)^2F&G&0\\
-w^2w^4F&G&(w^2)^2F&0\\
w^2w^3F-G&0&0&H
\end{array}\right].
\]
We next look at the conditions $k_{ijk}=k_{ikj}$.
From $k_{124}=k_{241}=0$ we find that
\[
w^3\fpd{(w^4F)}{w^4}=0.
\]
From $k_{224}=k_{242}=0$ we obtain
\[
(w^3)^2\fpd{F}{w^4}=0.
\]
It follows that $F=0$, except possibly where $w^3=0$ or $w^4=0$.
Thus $k_{11}=0$,  except possibly where $w^3=0$ or $w^4=0$; but then
by continuity $k_{11}=0$ everywhere; and similarly for the other
coefficients involving $F$. We are left with
\[
\left[\begin{array}{cccc}
0&0&0&-G\\
0&0&G&0\\
0&G&0&0\\
-G&0&0&H
\end{array}\right].
\]
This is evidently non-singular provided that $G$ is non-zero,
whatever $H$ may be. Continuing to analyse the consequences of the
condition $k_{ijk}=k_{ikj}$ we find that $G$ must be constant and
$H$ must be a function of $w^4$ alone. This gives as potential
Lagrangians
\[
l(w^1,w^2,w^3,w^4)=\lambda(w^2w^3-w^1w^4)+
\alpha_1w^1+\alpha_2w^2+\alpha_3w^3+h(w^4)
\]
where $\lambda$ and the $\alpha$\,s are constants with $\lambda$
non-zero, and $h$ is an arbitrary smooth function of its argument.
According to the general remarks made earlier, $l$ will in fact be a
Lagrangian if and only if $\alpha_1=\alpha_2=\alpha_3=0$ ($h$ doesn't
play a role here because $C^4_{ij}=0$).  It is easy to check this
directly.  In fact for the potential Lagrangian above it is easy to
see by direct calculation that all $\nu_i$ and almost all $\mu_{ij}$
vanish, except that $\mu_{23}=-\alpha_1$, $\mu_{24}=-\alpha_2$ and
$\mu_{34}=\alpha_3$ (and their skew counterparts).  We have shown
that there will exist a Lagrangian $l'=l+\theta_kw^k$ whose
Euler-Poincar\'e equations are exactly the equations associated to
$\gamma$ if we can find $\theta_k$ such that
$\mu_{ij}=\theta_kC^k_{ij}$.  One easily verifies that this condition
is only satisfied for $\theta_k=-\alpha_k$.  The sought-for Lagrangian
$l'$ is therefore the one above where one sets $\alpha_k=0$.

It is interesting to note that the most general Lagrangian in this
case is not just a quadratic form.

The method used by Muzsnay in \cite{Zolt} deals directly with the
equation
\[
w^jC^k_{ij}\fpd{l}{w^k}=0
\]
as a set of partial differential equations for $l$. In effect,
Muzsnay derives an integrability condition for this equation by
differentiating it, to obtain
\[
C^k_{ij}\fpd{l}{w^k}+w^lC^k_{il}\spd{l}{w^j}{w^k}=0.
\]
The part of this equation symmetric in $i$ and $j$ is our Helmholtz
condition (\ref{nabla}), the skew part states that the cocycle
$\mu_{ij}$ must vanish.  The examples we have considered above are two
of the many examples dealt with in \cite{Zolt}.  Muzsnay's results are
of course broadly the same as ours; however, in the second case though
he shows that a Lagrangian exists he does not indicate how to find
one, whereas we have obtained the most general one.  As Muzsnay points
out, the example is also treated in \cite{Gerard4dim}.  By using only
the unreduced Helmholtz conditions on $TG$, the authors of
\cite{Gerard4dim} look for a (not necessarily invariant) Lagrangian
$L$ for the canonical geodesic flow on any 4 dimensional Lie group.
Although they are not able to give an expression of the most general
Lagrangian, they observe in the case of the Lie algebra
$\mathrm{A}_{4,8}$ that the quadratic part of the Lagrangian above
(written in terms of invariant forms on $G$ in their set-up) generates
the flow of the canonical connection.  They also notice that the
quadratic part is a bi-invariant metric (as the theory predicts).

\subsection{The Bloch-Iserles equations}

These equations appear in e.g.\ \cite{BlochIserles,BlIsMaRa}. The
space of interest is $\Sym(n)$, the linear space of symmetric
$n\times n$ matrices. The equation is
\begin{equation}\label{BIeq}
\dot w = [w^2,N],
\end{equation}
where $w\in \Sym(n)$, $N$ is a skew-symmetric $n\times n$ matrix,
and the right-hand side is the commutator of matrices.  With the
help of $N$ one can give $\Sym(n)$ the structure of a Lie algebra,
the Lie algebra bracket being
\[
\{w_1,w_2\} = w_1 N w_2 - w_2 N w_1, \qquad\quad w_1,w_2\in \Sym(n).
\]
Can we find a Lagrangian $l\in \cinfty{\Sym(n)}$ for which
Equation~(\ref{BIeq}) is of Euler-Poincar\'e type with respect to the above Lie
algebra? The answer is in fact given in \cite{BlIsMaRa}:\ a
corresponding Lagrangian is
\begin{equation} \label{lagrangianBI}
l(w) = \onehalf \tr(w^2).
\end{equation}
We will show that the reduced Helmholtz conditions, applied to the
current Lie algebra and dynamical system, lead to the correct
Lagrangian.

To make things more accessible we will consider only the case
$n=2$. For a basis of the Lie algebra we take the matrices
\[
E_x = \left[\begin{array}{ll} 1 & 0 \\ 0 & 0\end{array}
\right],\quad E_y = \left[\begin{array}{ll} 0 & 1 \\
1 & 0\end{array} \right] \quad\mbox{and} \quad E_z =
\left[\begin{array}{ll} 0 & 0
\\ 0 & 1\end{array} \right] .
\]
Further, without loss of generality we can take $N$ to be
\[
\left[\begin{array}{rr} 0 & 1 \\ -1 & 0\end{array} \right].
\]
The non-vanishing Lie algebra brackets are then $\{E_x,E_y\} = 2E_z$,
$\{E_x,E_z\}=E_y$ and $\{E_y,E_z\}=2E_z$.
An arbitrary element of the Lie algebra is of the form
\[
w=x E_x + y E_y + z E_z =
\left[\begin{array}{ll} x & y \\
y & z\end{array} \right],
\]
and Equation (\ref{BIeq}) is
\[
\left[\begin{array}{ll}
\dot{x} & \dot{y} \\
\dot{y} & \dot{z}
\end{array}\right] =
\left[\begin{array}{cc}
-2 y(x+z) & x^2-z^2 \\
x^2-z^2 & 2y(x+z)
\end{array} \right].
\]
We use now the notation of the Lie algebroid formulation of the
Helmholtz conditions from Section~7. For $\Phi$  we find
\begin{eqnarray*}
\Phi(E_x) &=& (-3y^2 + \onehalf z^2) E_x + (\threehalf xy -2yz)E_y +
(4 y^2 -\onehalf xz) E_z,\\
\Phi(E_y) &=& (3xy-4yz) E_x +(4xz -\threehalf x^2- \threehalf z^2)  E_y
+(3yz - 4xy)E_z,\\
\Phi(E_z) &=& (4 y^2 -\onehalf xz) E_x+(\threehalf yz- 2xy) E_y +(-3y^2
+ \onehalf x^2) E_z,
\end{eqnarray*}
and for $\nabla$
\begin{eqnarray*}
\nabla E_x &=&  -(x+\onehalf z)E_y - y E_z, \\
\nabla E_y &=& (2x+z) E_x -(x + 2z)E_z, \\
\nabla E_z &=& y E_x + (z+ \onehalf x) E_y.
\end{eqnarray*}
The $\nabla$-equations in this case (taking the symmetry of $k_{ij}$
into account) are
\begin{equation} \label{nablaBI}
\begin{array}{l}
\gamma(k_{xx})+(2x+z)k_{xy}+ 2yk_{xz}=0, \\[1mm]
\gamma(k_{xy})+(x+\onehalf z)k_{yy}+yk_{yz}-(2x+z)k_{xx}+(x-2z)k_{xz}=0,\\[1mm]
\gamma(k_{xz})+(x+\onehalf z)k_{yz} +yk_{zz}-yk_{xx}
-(z+ \onehalf x)k_{xy}=0,\\[1mm]
\gamma(k_{yy})-2(2x+z)k_{xy}+2(x+2z)k_{yz}=0,\\[1mm]
\gamma(k_{yz})-(2x+z)k_{xz}
+(x+2z)k_{zz}-yk_{xy}-(z+\onehalf x)k_{yy}=0, \\[1mm]
\gamma(k_{zz})- 2yk_{xz} - (2z+x)k_{yz}=0.
\end{array}
\end{equation}
The $\Phi$-equations are
\begin{equation} \label{PhiBI}
\begin{array}{l}
k_{xx}(3yx-4yz)+k_{xy}(-\frac{3}{2}
x^2+4xz-\frac{3}{2}z^2)+k_{xz}(-4yx+3yz)
\\[1mm] \qquad = k_{xy}(-3y^2+\onehalf y^2)+k_{yy}(\frac{3}{2} yx-2yz)
+k_{yz}(4y^2-\onehalf xz),\\[2mm]
k_{xx}(4y^2- \onehalf xz)+k_{xy}(-2yx+\frac{3}{2}yz)+k_{xz}(-3y^2+\onehalf
x^2)\\[1mm]
\qquad =
k_{xz}(-3y^2+ \onehalf z^2)+k_{yz}(\frac{3}{2}xy-2yz)+k_{zz}(4y^2-\onehalf
xz),\\[2mm]
k_{xy}(4y^2-\onehalf
xz)+k_{yy}(-2xy+\frac{3}{2}yz)+k_{yz}(-3y^2+\onehalf
x^2)\\[1mm]
\qquad =
k_{xz}(3xy-4yz)+k_{yz}(-\frac{3}{2}x^2+4xz-\frac{3}{2}z^2)+k_{zz}(-4xy+3yz).
\end{array}
\end{equation}

We will first try to find a solution of (\ref{nablaBI}) in which
all $k_{ij}$ are constants. In that case, adding one half times the
$(y,y)$-equation to the $(x,x)$- and $(z,z)$-equations gives
$k_{xy}=0$ from which also $k_{xz}=0$ and $k_{yz}=0$. Then the
$(x,z)$-equation gives $k_{xx}=k_{zz}$ and the $(y,z)$-equation
gives $k_{zz}=\frac{1}{2}k_{yy}$. So the solutions of
(\ref{nablaBI}) with constant coefficients are of the form
\begin{equation}\label{kBI}
k = c \left[ \begin{array}{ccc} 1 & 0 & 0 \\ 0&2&0 \\ 0&0&1
\end{array} \right].
\end{equation}
It is easy to see that a multiplier $k$ of this form also satisfies
Equations (\ref{PhiBI}). The Hessian of the function
\[
l(x,y,z) = \onehalf(x^2+2y^2+z^2)
\]
takes the above form; this is exactly the Lagrangian (\ref{lagrangianBI}).

An expression for the most general solution of Equations
(\ref{nablaBI}) and (\ref{PhiBI}) is beyond the scope of the current
paper. However, instead of looking for constant solutions $k_{ij}$
as above, we could use an additional symmetry assumption. For
example, it seems natural to require that $k_{xx}=k_{zz}$ (but that
they are not necessarily constants). A tedious calculation reveals
that in that case the only possible solution of the reduced
Helmholtz conditions is again the multiplier in (\ref{kBI}) with
constant coefficients.

\subsection{An illustrative example on the Lie group of the
affine line}

There are only two distinct Lie algebras of dimension 2. In
this example we will use the Lie group of the affine line (the
Euclidean group). An element of this group is an affine map
$\R\to\R: t\mapsto\exp( q_1)t+q_2$ and can be represented by the
matrix
\[\left[
\begin{array}{ll} \exp(q_1) & q_2 \\ 0 & 1 \end{array}\right].
\]
The corresponding Lie algebra is given by the set of matrices of the
form
\[
\left[ \begin{array}{ll} x & y \\ 0 & 0 \end{array}\right].
\]
A basis for this algebra is
\[
E_x=\left[\begin{array}{cc} 1 & 0 \\ 0
& 0\end{array}\right], \qquad
E_y= \left[\begin{array}{cc} 0 & 1 \\
0 &0\end{array}\right]
\]
for which $\{E_x,E_y\}=E_y$. Let $A=a E_x + b E_y$ be a constant
vector in the Lie algebra. We will determine whether there exists a
regular Lagrangian for the dynamical system
\[
\dot w = \{w,\{w,A\}\},
\]
or, in the above basis,
\[
\dot{x}=0, \qquad \dot y= x(bx-ay).
\]
For this system
\[
\begin{array}{ll}
\Phi(E_x) = \frac{1}{4}(a-1)^2 xy E_y,
&\quad \nabla E_x = (-bx+\frac{1}{2}(a-1)y)E_y, \\[2mm]
\Phi(E_y)=-\frac{1}{4}(a-1)^2x^2 E_y , &\quad \nabla E_y =
\frac{1}{2}(a+1)x E_y.
\end{array}
\]
The $\nabla$-equations are therefore
\[
\begin{array}{l}
x(bx-ay)k_{xxy}-2(-xb+\onehalf(a-1)y) k_{xy} = 0, \\[2mm]
x(bx-ay)k_{xyy}- \onehalf (a+1)x k_{xy}-(-bx+\onehalf(a-1)y) k_{yy}
=0,\\[2mm] x(bx-ay)k_{yyy}-  (a+1)x k_{yy} = 0,
\end{array}
\]
and the only $\Phi$-equation is
\[
-(a-1)^2x^2k_{xy} = (a-1)^2xyk_{yy}.
\]
If we differentiate the $\Phi$-equation with respect to $x$ and $y$ we
obtain two more equations for the $k_{ijk}$:
\[
\begin{array}{l} -(a-1)^2(x^2k_{xxy}
+2xk_{xy})=(a-1)^2(xyk_{xyy}+
 y k_{yy}),\\[2mm]
-(a-1)x^2k_{xyy} = (a-1)^2(xyk_{yyy} + xk_{yy}).
\end{array}
\]
The component $k_{xx}$ of the Hessian and its derivative $k_{xxx}$
are absent from these equations, and they will also not show up in
any derived equation; there will therefore always remain freedom of
choice for the $x$-derivative of $k_{xx}$. We get 6 homogeneous
linear equations in the 5 unknowns $k_{xy}$, $k_{yy}$, $k_{xxy}$,
$k_{xyy}$ and $k_{yyy}$.  If the rank of this system is less than 5
the system will have a non-zero solution. When $a=1$, the rank is
clearly 3. It can easily be verified that in all other cases the
rank is 4.

For reasons of clarity, we will deal first with the case where
$a=1$.

{\bf 1. The case where $a=1$}. In this case the $\Phi$-equation is
identically satisfied. The $\nabla$-equations are now
\[
\begin{array}{l}x (bx - y) k_{xxy} + 2 bx k_{xy} = 0, \\
 x(bx-y)k_{xyy}-xk_{xy}+bx k_{yy} = 0, \\
x(bx-y)k_{yyy}-2xk_{yy} = 0. \end{array}
\]
From the first and the last of these equations, we find that
\[
k_{xy}= \frac{f_1(y)}{(bx-y)^2}\qquad\mbox{and}\qquad
k_{yy}=\frac{f_2(x)}{(bx-y)^2}
\]
respectively, as long as $x\neq 0$ and $bx-y\neq 0$. By substituting
this result in the second equation and by interpreting $k_{xyy}$
once as $\partial k_{yy}/\partial x$ and once as $\partial
k_{xy}/\partial y$, we get a system of ODE's from which we can
determine $f_1(y)$ and $f_2(x)$. They are
\[
f_1(y)=-b\alpha_2 -\alpha_1 y\qquad\mbox{and}\qquad f_2(x)=\alpha_1
x +\alpha_2.
\]
The solution of the $\nabla$-equations is therefore of the form
\[ k= \left[\begin{array}{cc}
\frac{b^2\alpha_2-b\alpha_1(bx-y)}{2(bx-y)^2}+f(x) &
-\frac{b\alpha_2+\alpha_1y}{(bx-y)^2}\\[2mm]
-\frac{b\alpha_2+\alpha_1y}{(bx-y)^2} &
\frac{\alpha_1x+\alpha_2}{(bx-y)^2}\end{array} \right],\qquad
\det(k)=\frac{f(x)(\alpha_1x+\alpha_2)-\alpha_1^2}{(bx-y)^2}.
\]
This matrix is, however, not defined on the whole of $\R^2$. By
continuity it exists on $x=0$ but it is not defined on $bx-y=0$. So,
there is no regular multiplier on $\R^2$.

This is not the end of the story, however. The dynamical equations
are now $\dot x=0$ and $\dot y=x(bx-y)$. Notice that the lines $x=0$
and $bx-y=0$ are both invariant under the flow. They divide the
space $\R^2$ into regions, each invariant under the flow of the
dynamical system. The matrix above is well-defined on the invariant
region with $bx-y\neq 0$.  It will be a multiplier provided its
determinant is not zero, that is, provided
$f(x)(\alpha_1x+\alpha_2)-\alpha_1^2\neq0$. The function
$l(x,y)=-\alpha_2\ln|bx-y|-\alpha_1x\ln|bx-y|+ \alpha_3y+h(x)$, with
$h''(x)=f(x)$, has the above matrix as its Hessian.  However, for
$l$ to give the required Euler-Poincar\'e equations, $\alpha_2$ and
$\alpha_3$ must vanish.  A non-degenerate Lagrangian on $bx-y\neq0$
is therefore
\[
l(x,y)=-\alpha_1x\ln|bx-y|+h(x),
\]
where $\alpha_1$ is a non-vanishing constant, and $h(x)$ is an
arbitrary function which is not of the form
$\alpha_1(x\ln|x|-x)+\alpha_4 x + \alpha_5$ for any
constants $\alpha_4$ and $\alpha_5$.

In the following cases it will happen that there is no regular
Lagrangian defined on the whole of $\R^2$, but it may be possible to
find Lagrangians for subsets of $\R^2$ invariant under the dynamical
flow.

{\bf 2.  The case where $a\neq 1$.} In this case, the
$\Phi$-equations come into play. As before, we can search first for
the most general class of solutions of the $\nabla$-equations, and
then restrict to only those that also satisfy the $\Phi$-equation.
Notice that e.g.\ the last of the $\nabla$-equations leads to a
further division of this case in subcases. We have
\[
k_{yy} = \left\{ \begin{array}{ll} f_2(x) (bx-ay)^{-\frac{1}{a}-1},
&\qquad a\neq 0,
\\[2mm]\displaystyle
f_2(x) \exp\left(\frac{y}{bx}\right),  &\qquad a=0,\, b\neq 0, \\[2mm] 0,
&\qquad a=0,\, b=0.
\end{array}   \right.
\]
We will only summarize the results.

{\em 2A. The case where $a\neq 0$.} There is a regular Lagrangian of
the form
\[
l(x,y)= \frac{\alpha_1}{1-a}|ay-bx|^{1-\frac{1}{a}}|x|^{\frac{1}{a}}
+ h(x),
\]
where $\alpha_1$ is a non-zero constant and $h(x)$ is an arbitrary,
but non-affine function. The lines $x=0$ and $ay-bx=0$ are invariant
under the flow.

{\em 2B. The case where $a=0$ and $b\neq 0$.} There is a regular
Lagrangian of the form
\[
l(x,y) = \alpha_1 b^2 x \exp\left(\frac{y}{bx}\right) + h(x),
\]
where $\alpha_1$ is a non-zero constant and $h(x)$ is an arbitrary
but non-affine function. The line $x=0$ is invariant under the flow.

{\em 2C. The case where $a=0$ and $b = 0$.} This is a degenerate
case, there is no regular multiplier.

Having decided in all cases whether a Lagrangian $l$ on the Lie
algebra $\g$ exists or not, it is instructive to give an expression
for the corresponding Lagrangians $L$ at the level of the Lie group
$G$. If $(q_1,q_2)$ are coordinates on the Lie group, then a
left-invariant basis of vector fields is given by
\[
{\hat E}_x = \fpd{}{q_1}, \qquad {\hat E}_y = \exp(q_1)\fpd{}{q_2}.
\]
Fibre coordinates $(w^i)=(x,y)$ with respect to this basis and $({\dot
q}_1,{\dot q}_2)$ with respect to the coordinate field basis are
related as $x={\dot q}_1, y= \exp(-q_1){\dot q}_2$.  A right-invariant
basis is
\[
{\tilde E}_x = \fpd{}{q_1} + q_2\fpd{}{q_2}, \qquad {\tilde E}_y =
\fpd{}{q_2}.
\]
The complete and vertical lifts of the left-invariant basis fields are
\[
\clift{\hat E}_x = \fpd{}{q_1}, \qquad \clift{\hat E}_y =
\exp(q_1)\left(\fpd{}{q_2} +{\dot q}_1\fpd{}{{\dot q}_2}\right),\qquad
\vlift{\hat E}_x = \fpd{}{{\dot q}_1}, \qquad \vlift{\hat E}_y =
\exp(q_1)\fpd{}{{\dot q}_2}.
\]
We can now rewrite a second-order field $\Gamma$ in any of the
following forms
\begin{eqnarray*}
\Gamma &=& {\dot q}_1 \fpd{}{q_1} + {\dot q}_2 \fpd{}{q_2}
+ f_1\fpd{}{{\dot q}_1} + f_2 \fpd{}{{\dot q}_2} \\
&=&  {\dot q}_1\fpd{}{q_1} +  \exp(-q_1) {\dot q}_2
\left(\exp(q_1)\left(\fpd{}{q_2}+{\dot q}_1\fpd{}{{\dot q}_2}\right)\right)\\
&&\mbox{} + f_1 \fpd{}{{\dot q}_1}
+ \left( \exp(-q_1)(f_2 - {\dot q}_1{\dot q}_2)\right)\exp(q_1)\fpd{}{{\dot q}_2}\\
&=& x \clift{\hat E}_x + y \clift{\hat E}_y + \Gamma_x  \vlift{\hat E}_x
+ \Gamma_y \vlift{\hat E}_y.
\end{eqnarray*}
In the example under consideration, $\Gamma_x=0$ and $\Gamma_y=
x(bx-ay)$, so
\[
f_1= 0, \qquad f_2 = (1-a){\dot q}_1{\dot q}_2 + \exp(q_1)b {\dot
q}_1^2).
\]

Let's look, for example, at case 2B ($a=0$), where we have stated
above that there exist a regular Lagrangian on the Lie algebra of
the form $l(x,y) = \alpha_1 b^2 x \exp(y/bx) + h(x)$ ($\alpha_1\neq
0$, $h$ non-affine). By using left translations we can extend this
to a Lagrangian on the whole of $TG$:
\[
L(q_1,q_2,{\dot q}_1,{\dot q}_2) = \alpha_1 b^2 {\dot q}_1
\exp\Big(\frac{ \exp(-q_1){\dot q}_2}{b {\dot q}_1}\Big) + h({\dot
q}_1).
\]
Obviously, this Lagrangian is invariant:
\[
\clift{{\tilde E}_1}(L) = \fpd{L}{q_1}+q_2\fpd{L}{q_2}+ {\dot
q}_2\fpd{L}{{\dot q}_2}=0\quad \mbox{and}\quad \clift{{\tilde
E}_2}(L) = \fpd{L}{q_2}=0.
\]
A short calculation shows that the Euler-Lagrange equations for the
above Lagrangian do indeed return the differential equations ${\ddot
q}_1=0$ and ${\ddot q}_2 = {\dot q}_1{\dot q}_2 + b\exp(q_1) {\dot
q}_1^2$, as they should.

Our analysis reveals  only whether there is an invariant Lagrangian.
In the case 2C ($a=b=0$) where no such Lagrangian exists there could
still be a (necessarily non-invariant) Lagrangian for the
second-order system ${\ddot q}_1 =0$, ${\ddot q}_2 = {\dot q}_1{\dot
q}_2$ on the two-dimensional Lie group.  In \cite{Douglas} Douglas
gave a more-or-less complete classification of the inverse problem
for two-dimensional systems. A modern geometric approach to
Douglas's classification can be found in \cite{WillyTrans}.  A
meticulous analysis using the methods described there shows that a
regular Lagrangian must exist, even in the case 2C where we
concluded that there is no invariant Lagrangian. In more detail, our
case~1 belongs to Douglas's case~I, and our cases 2A, 2B and 2C to
his case~IIa1.

Observe that if $a=b=0$ we are back in the example of the canonical
connection. According to \cite{Thompson1}, the most general
Lagrangian for the case 2C, subject to the regularity condition, is
given by
\[
L(q_1,q_2,{\dot q}_1,{\dot q}_2) = {\dot q}^1\theta(q_1,q_2,z) +
\psi({\dot q}^1) , \qquad z={\dot q}_2/{\dot q}_1,
\]
where $\psi$ is an arbitrary function and $\theta$ is a solution of
the PDE
\[
z\theta_{zz} + z\theta_{zq_2} +\theta_{q_1z}-\theta_{q_2} =0
\]
(subscripts denote derivatives, as usual). For example, the function
\[
L(q_1,q_2,{\dot q}_1,{\dot q}_2) = \onehalf {\dot q}_1^2 +
\exp(-q_1)\frac{{\dot q}_2^2}{2{\dot q}_1}
\]
is a Lagrangian for the system in 2C. It is clearly not invariant
since
\[
\clift{\tilde E}_1(L) = \exp(-q_1)\frac{{\dot q}_2^2}{2{\dot q}_1}.
\]
In fact, there does not exist a function $\theta$ for which the
Lagrangian is invariant and regular. The relations $\clift{\tilde
E}_1(L)=0$ and $\clift{\tilde E}_2(L)=0$ imply that $\theta_{q_1} + z
\theta_z=0$ and $\theta_{q_2} =0$, respectively. By taking the
$z$-derivative of the first relation and by applying the second in
the defining relation of $\theta$, we can conclude that also
$\theta_z=0$ and $\theta_{q_1}=0$. But then $\theta$ is a constant
and the Lagrangian is clearly degenerate.

\section{Outlook}

We discuss briefly two possible extensions of the current framework.
First of all, let $M$ be a manifold with a given symmetry group $G$.
One can then set up an inverse problem for $G$-invariant Lagrangians
on $M$. In that case, it has been shown in \cite{Cendra} that the
Euler-Lagrange equations reduce to the so-called {\em
Lagrange}-Poincar\'e equations. On the other hand, the technique of
adapted frames can be easily extended to manifolds with a symmetry
group; a description of the reduced equations for arbitrary
second-order equations can be found in \cite{MikeTom}. So the
question would be when these reduced equations are of
Lagrange-Poincar\'e form.

The second extension can be situated at the level of Lie algebroids.
In this paper, we have discussed the inverse problem for Lagrangians
on a Lie algebra $\g$. The original inverse problem deals with
Lagrangians on $TM$. Both $\g$ and $TM$ are the two simplest cases
of a Lie algebroid. So it seems natural to study an inverse problem
for arbitrary Lie algebroids (the corresponding Lagrangian equations
were given in e.g.\ \cite{Ed}). The situation in the previous
paragraph then coincides with the case that the Lie algebroid is
$TM/G$, the so-called Atiyah algebroid.

\subsubsection*{Acknowledgements}

We are deeply indebted to Eduardo Mart\'inez for his essential
contribution to the discussion about the cohomology conditions.  We
also wish to thank Gerard Thompson for many helpful comments about the
inverse problem for the canonical connection, and Willy Sarlet for
going through Example 7.3 using the methods of \cite{WillyTrans}.

The first author is a Guest Professor at Ghent University:\ he is
grateful to the Department of Mathematical Physics and Astronomy at
Ghent for its hospitality.

The second author is currently a Research Fellow at The University
of Michigan through a Marie Curie Fellowship. He is grateful to the
Department of Mathematics for its hospitality. He also acknowledges
a research grant (Krediet aan Navorsers) from the Fund for
Scientific Research - Flanders (FWO-Vlaanderen), where he is an
Honorary Postdoctoral Fellow.


\begin{thebibliography}{99}

\bibitem{AndPoh} I.\,M.\ Anderson and J.\ Pohjanpelto, The cohomology
of invariant variational bicomplexes, {\em Acta Appl.\ Math.} {\bf
41} (1995), 3-19.

\bibitem{AndThomp} I.\,M.\ Anderson and G.\ Thompson, The inverse problem of
the calculus of variations for ordinary differential equations, {\em
Mem.\ Amer.\ Math.\ Soc.} {\bf 98} (1992).

\bibitem{BlochIserles} A.\,M.\ Bloch and A.\ Iserles, On an
isospectral Lie-Poisson system and its Lie algebra, {\em Found.\
Comput.\ Math.} {\bf 6} (2006) 121-144.

\bibitem{BlIsMaRa} A.\,M.\ Bloch, A.\ Iserles, J.\,E.\ Marsden and
T.\,S.\ Ratiu, A class of integrable geodesic flows on the symplectic group
and the symmetric matrices, arXiv:math-ph/0512093.

\bibitem{Cendra} H.\ Cendra, J.\,E.\ Marsden and T.\,S.\ Ratiu, Lagrangian
reduction by stages, {\em Mem.\ Amer.\ Math.\ Soc.} {\bf 152} (2001).

\bibitem{Crampin} M.\ Crampin, On the differential geometry of the
Euler-Lagrange equations, and the inverse problem of Lagrangian
dynamics, {\em J.\ Phys.\ A:\ Math.\ Gen.} {\bf 14} (1981) 2567--2575.

\bibitem{Towards} M.\ Crampin, W.\ Sarlet, G.\,B.\ Byrnes and G.\,E.\ Prince,
Towards a geometrical understanding of Douglas's solution of the
inverse problem of the calculus of variations, {\em Inverse
Problems} {\bf  10} (1994) 245-260.

\bibitem{MikeTom} M.\ Crampin and T.\ Mestdag, Reduction and
reconstruction aspects of second-order dynamical systems with
symmetry, preprint (2006), available at maphyast.ugent.be.

\bibitem{Douglas} J.\ Douglas, Solution of the inverse problem of
the calculus of variations, {\em Trans.\ Amer.\ Math.\ Soc.} {\bf
50} (1941) 71--128.

\bibitem{Gerard4dim} R.\ Ghanam,  G.\ Thompson and E.\,J.\ Miller,
Variationality of four-dimensional Lie group connections, {\em J.\
Lie Theory\/} 14 (2004) 395--425.

\bibitem{Henn}
M.\ Henneaux and L.\,C.\ Shepley, Lagrangians for spherically
symmetric potentials, {\em J.\ Math.\ Phys.} {\bf 23} (1982)
2101--2107.


\bibitem{KP} O.\ Krupkov\'a and G.\,E.\ Prince, Second-order ordinary
differential equations in jet bundles and the inverse problem of the
calculus of variations, Chapter 16 of D.\ Krupka and D.\,J.\ Saunders (eds.),
{\em Handbook of Global Analysis\/}, Elsevier (2007).

\bibitem{Marmo2} G.\ Marmo and G.\ Morandi, On the inverse problem
with symmetries, and the appearance of cohomologies in classical
Lagrangian dynamics, {\em Rep.\ Math.\ Phys.} {\bf 28} (1989)
389--410.

\bibitem{MR} J.E.\ Marsden and T.\ Ratiu, {\em Introduction to
Mechanics and Symmetry\/}, Texts in Applied Mathematics 17,
Springer (1999).

\bibitem{Ed} E.\ Mart\'\i nez, Lagrangian mechanics on Lie algebroids,
{\em Acta.\ Appl.\ Math.} {\bf 67} (2001) 295--320.

\bibitem{MCS} E.\ Mart\'{\i}nez, J.\,F.\ Cari\~{n}ena and W.\ Sarlet,
Derivations of differential forms along the tangent bundle projection II,
{\em Diff.\ Geom.\ Appl.} {\bf 3} (1993) 1--29.

\bibitem{OldMarmo} G.\ Morandi, C.\ Ferrario, G.\ Lo Vecchio, G.\ Marmo
and C.\ Rubano,
The inverse problem in the calculus of variations and the geometry
of the tangent bundle, {\em Physics Reports\/} {\bf 188} (1990) 147-284.

\bibitem{Zolt} Z.\ Muzsnay, An invariant variational principle for
canonical flows on a Lie group, {\em J.\ Math.\ Phys.} {\bf 46} (2005) 112902.

\bibitem{Pat}
J.\ Patera, R.\,T.\ Sharp, P.\ Winternitz and H.\ Zassenhuis,
Invariants of real low dimension Lie algebras, {\em J.\ Math.\ Phys.}
{\bf 17} (1976) 986--994.

\bibitem{Thompson2} M.\ Rawashdeh and G.\ Thompson, The inverse
problem for six-dimensional codimension two nilradical Lie algebras,
{\em J.\ Math. Phys.} {\bf 47} (2006) 112901.

\bibitem{Santilli} R.\,M.\ Santilli, {\em Foundations of Theoretical Mechanics},
Springer (1978).


\bibitem{WillyTrans} W. Sarlet, G. Thompson and G.\,E. Prince,
The inverse problem in the calculus of variations: the use of geometrical
calculus in Douglas's analysis, {\em Trans.\ Amer.\ Math.\ Soc.} {\bf 354}
(2002) 2897--2919.

\bibitem{Thompson1} G.\ Thompson, Variational connections on Lie
groups, {\em Diff.\ Geom.\ Appl.} {\bf 18} (2003) 255--270.


\end{thebibliography}
\end{document}